\documentclass{aims}
\usepackage[pagewise]{lineno}
\usepackage{amsmath}
\usepackage{paralist}
\usepackage[misc]{ifsym}
\usepackage{epsfig} 
\usepackage{epstopdf} 
\usepackage[colorlinks=true]{hyperref}
\hypersetup{urlcolor=blue, citecolor=red}
\allowdisplaybreaks

\def\currentvolume{X}
 \def\currentissue{X}
  \def\currentyear{200X}
   \def\currentmonth{XX}
    \def\ppages{X-XX}
     \def\DOI{10.3934/xx.xxxxxxx}

\usepackage{bm}
\usepackage{accents}

\usepackage{cleveref}
\usepackage{tikz}
\usetikzlibrary{spy}
\usetikzlibrary {decorations.fractals,spy}
\usepackage{placeins}
\usepackage[colorinlistoftodos,prependcaption]{todonotes}
\usepackage{xcolor}
\usepackage{booktabs}
\usepackage[export]{adjustbox}
\usepackage{multirow}

\newcommand{\codecomment}[1]{{\color{gray} // #1}}

\renewcommand{\t}{^{\!\top}}
\newcommand{\K}{\mathcal{K}}

\newcommand{\R}{\mathbb{R}}
\newcommand{\ch}{\mathrm{ch}}
\newcommand{\pr}{\mathrm{pr}}

\renewcommand{\exp}{\mathrm{exp}}
\newcommand{\curr}{\mathrm{curr}}

\newcommand{\floor}[1]{\left\lfloor #1 \right\rfloor}

\DeclareMathOperator*{\argmax}{arg\,max}

\newcommand{\bigO}[1]{\mathcal{O}\!\left(#1\right)}

\renewcommand{\S}{\mathcal{S}}
\newcommand{\ke}{k_{\exp}}
\newcommand{\kc}{k_{\ch}}
\newcommand{\Se}{S_{\exp}}
\newcommand{\Sc}{S_{\ch}}
\newcommand{\bSe}{\bS_{\exp}}
\newcommand{\bSc}{\bS_{\ch}}
\newcommand{\tSe}{\widetilde{S}_{\exp}}
\newcommand{\tSc}{\widetilde{S}_{\ch}}

\newcommand{\bAe}{\bA_{\exp}}
\newcommand{\bAc}{\bA_{\ch}}
\newcommand{\sigmaE}{\sigma_{\exp}}
\newcommand{\sigmaC}{\sigma_{\ch}}
\newcommand{\bfae}{\bfa_{\exp}}
\newcommand{\bfac}{\bfa_{\ch}}
\newcommand{\cC}{c_\ch}
\newcommand{\cE}{c_\exp}
\newcommand{\bud}{b}

\newcommand{\epsII}{\epsilon}

\newcommand{\bA}{\mathbf{A}}
\newcommand{\bB}{\mathbf{B}}

\newcommand{\bD}{\mathbf{D}}

\newcommand{\bI}{\mathbf{I}}

\newcommand{\bS}{\mathbf{S}}

\newcommand{\bU}{\mathbf{U}}
\newcommand{\bV}{\mathbf{V}}

\newcommand{\bY}{\mathbf{Y}}

\newcommand{\bPhi}{\mathbf{\Phi}}
\newcommand{\bPsi}{\mathbf{\Psi}}

\newcommand{\bSigma}{\mathbf{\Sigma}}

\newcommand{\bfa}{\mathbf{a}}

\newcommand{\bfe}{\mathbf{e}}

\newcommand{\bfh}{\mathbf{h}}

\newcommand{\bfm}{\mathbf{m}}

\newcommand{\bfu}{\mathbf{u}}
\newcommand{\bfv}{\mathbf{v}}

\newcommand{\bfy}{\mathbf{y}}

\newcommand{\bfeps}{\boldsymbol{\varepsilon}}
\renewcommand{\b}[1]{\bm{#1}}

\makeatletter
\newcommand{\algorithmicparfor}{\textbf{parfor}}
\newcommand{\PARFOR}[1]{\ALC@it\algorithmicparfor\ #1\ \algorithmicdo}
\newcommand{\ENDPARFOR}{\ALC@it\algorithmicend\ \algorithmicparfor}
\makeatother

\usepackage{algorithm}
\usepackage{subcaption}
\usepackage{algorithmic}
\usepackage{amssymb}
\newcommand{\myparagraph}[1]{%
  \par\medskip
  \noindent\textbf{#1.}\quad
}

\newtheorem{theorem}{Theorem}[section]
\newtheorem{corollary}[theorem]{Corollary}

\newtheorem{lemma}[theorem]{Lemma}
\newtheorem{proposition}[theorem]{Proposition}

\theoremstyle{definition}

\title[Multifidelity sensor placement]
{Multifidelity sensor placement in Bayesian state estimation problems} 

\author[G. Ramon, G. Sarnoski, V. Tumuluri, H. D\'iaz and A. K. Saibaba  ]{}


\thanks{$^*$Corresponding author: Hugo D\'iaz}

\begin{document}
\maketitle

\centerline{\scshape
Gabriela Ramon$^{{\href{mailto:gabriela.ramon1@marist.edu}{\textrm{\Letter}}}1}$

Geena Sarnoski$^{{\href{mailto:gsarnos1@students.towson.edu}{\textrm{\Letter}}}2}$

Vasishta Tumuluri$^{{\href{mailto:vtumulur@unc.edu}{\textrm{\Letter}}}3}$

Hugo D\'iaz$^{{\href{mailto:hdiazno@ncsu.edu}{\textrm{\Letter}}}*4}$

Arvind K.\ Saibaba$^{{\href{mailto:asaibab@ncsu.edu}{\textrm{\Letter}}}4}$
}

\medskip

{\footnotesize
 \centerline{$^1$School of Computer Science and Mathematics, Marist University, Poughkeepsie, NY, USA}
}

\medskip

{\footnotesize
 \centerline{$^2$Department of Mathematics, Towson University, Towson, MD, USA}
}
\medskip
{\footnotesize
 \centerline{$^3$Department of Statistics \& Operations Research, University of North Carolina, Chapel Hill, NC, USA}
}

\medskip

{\footnotesize
 \centerline{$^4$Department of Mathematics, North Carolina State University, NC, USA}
}

\bigskip

 \centerline{(Communicated by Handling Editor)}


\begin{abstract}
We study optimal sensor placement for Bayesian state estimation problems in which sensors vary in cost and fidelity, resulting in a budget-constrained multifidelity optimal experimental design problem. 
Sensor placement optimality is quantified using the D-optimality criterion, and the problem is approached by leveraging connections with the column subset selection problem in numerical linear algebra.
We implement a greedy approach for this problem, whose computational efficiency we improve using rank-one updates via the Sherman-Morrison formula. 
We additionally present an iterative algorithm that, for each feasible allocation of sensors, greedily optimizes over each sensor fidelity subject to previous sensor choices, repeating this process until a termination criterion is satisfied. 
To our knowledge, these algorithms are novel in the context of cost-constrained multifidelity sensor placement.
We evaluate our methods on several benchmark state estimation problems, including reconstructions of sea surface temperature and flow around a cylinder, and empirically demonstrate improved performance over random designs.
\end{abstract}


\section{Introduction}

In many inverse problems, from climate and weather prediction to medical imaging and fluid dynamics, the effectiveness of data collection strategies strongly influences the quality of the resulting estimates \cite{ocean_climate_observing, tsunami, mri_imaging, fluid1, fluid2}. Within these problems, the underlying system state is often only observable through noisy measurements, and it can be expensive or harmful to take many such measurements. Due to these budgetary and practical considerations, practitioners often rely on sparse and noisy measurements to reconstruct the state. Therefore, deciding where to place sensors and what type of sensors to use is a central challenge in these applications.

This challenge is often referred to as the sparse sensor selection problem, in which one seeks the most informative sensor configuration for reconstructing a system state from limited data \cite{9152984}. Previous work has examined this problem, assuming a fixed number of sensors of the same quality \cite{Attia_2022, eswar2025bayesiandoptimalexperimentaldesigns, article, kakasenko2025bridginggapdeterministicprobabilistic}. However, in practice, sensors often vary widely in quality and cost: high-precision sensors may be more accurate but significantly more expensive or power-hungry, while cheaper alternatives may provide noisier or biased measurements \cite{evaluation_metrics,biosensors,structural_damage}. These trade-offs highlight the need for multifidelity, cost-constrained sensor placement strategies, which have received comparatively little attention (see Related Work in \Cref{ssec:relatedwork}).

In this work, we define sensor fidelity as a function of both financial cost and signal-to-noise ratio (sensor variance).
We consider the multifidelity problem in which we have to place sensors with two different fidelities---cheap (less expensive with high variance) and expensive (more expensive with low variance).
Our objective is to determine the optimal spatial configuration of these sensors without exceeding a fixed global budget.
To determine the optimal placement, we adopt a Bayesian experimental design framework and aim to maximize the D-optimality criterion~\cite{alexanderian2023briefnotebayesiandoptimality}, which is equivalent to minimizing the uncertainty of the reconstruction and maximizing the expected information gain from the prior to the posterior distribution.
Because this exact 
optimization problem is known to be NP-hard (see, e.g.,~\cite{eswar2025bayesiandoptimalexperimentaldesigns}), we focus on  developing 
approximation 
algorithms.

In this paper, we introduce a greedy approach to this problem based on marginal gain 
per unit of cost, extending existing greedy approaches from the single-fidelity case. We formulate a computationally efficient implementation of this algorithm using symmetric rank-one updates upon the addition of each sensor. Additionally, we improve upon the greedy approach with an iterative selection algorithm that is novel in the context of multifidelity sensor placement and shows stronger empirical results, in terms of D-optimality, in our experiments.

\subsection{Outline and Contributions}
We summarize our main contributions as follows:
\begin{enumerate}

\item \textbf{Multifidelity Sensor Placement} (discussed in \Cref{sec:mfgreedy}): 
We formulate the cost-constrained multifidelity optimal sensor placement problem and draw connections to the knapsack problem and the column subset selection problem. 
\item {\textbf{Theoretical Properties} (discussed in \Cref{subsec:properties})}: 
We establish theoretical properties of the D-optimality objective function for our multifidelity sensor placement problem and note the advantages of a greedy approach in this context.
\item \textbf{Greedy Algorithm} (discussed in \Cref{subsec:mfgreedy,subsec:sherman,subsec:performance}): 
We extend the classical greedy sensor placement methods from the single-fidelity setting to a multifidelity framework using a heuristic based on marginal improvement of the objective normalized by sensor cost. We derive a computationally efficient implementation using rank-one updates via the Sherman--Morrison formula and provide a detailed complexity analysis. Our analysis shows that the proposed implementation improves computational efficiency relative to existing greedy approaches, but also that our adapted greedy algorithm does not admit a constant-factor approximation guarantee.

\item \textbf{Iterative Selection} (discussed in \Cref{sec:iterative}): 
To overcome the drawbacks of the greedy algorithm, we propose an iterative algorithm that alternates greedy optimization across sensor fidelities, refining the selection at each step until a termination criterion is met.
To the best of our knowledge, this iterative refinement strategy has not been previously studied in the context of cost-constrained multifidelity sensor placement.
We derive a computationally efficient implementation by exploiting structural properties of the Bayesian D-optimality objective and provide a corresponding complexity analysis.

\item \textbf{Numerical Experiments} (discussed in \Cref{sec:experiments}):
We compare our greedy and iterative selection algorithms against random designs using benchmark state estimation problems, including the reconstruction of sea surface temperature and flow around a cylinder. We explore the behavior of the greedy algorithm across different problem instances and theoretically corroborate our experimental findings via a Taylor approximation of the marginal improvement. Our experimental results indicate that the iterative approach performs at least as well as the greedy algorithm in terms of the D-optimality criterion, and that both perform considerably better than baseline random designs.
\end{enumerate}

The remainder of this paper is structured as follows.
In \Cref{sec:prelim}, we review relevant preliminaries from single-fidelity sensor placement and linear algebra. In \Cref{sec:mfgreedy}, we introduce the multifidelity sensor placement problem and suggest a greedy approach to find approximately optimal sensor locations, and in \Cref{sec:iterative}, we propose an iterative algorithm that repeatedly applies greedy selection across each fidelity. In \Cref{sec:experiments}, we present experimental results of our approaches applied to benchmark state estimation problems.

\subsection{Related Work}\label{ssec:relatedwork}
D-optimal sparse sensor selection has been well-studied in the single-fidelity case. Although finding an exact solution to this binary optimization problem is NP-hard~ \cite{eswar2025bayesiandoptimalexperimentaldesigns}, previous work has suggested various approaches to find approximately optimal solutions. 
Greedy approaches are quite popular for D-optimal sensor placement due to their ease of implementation and theoretical guarantees stemming from the submodularity of the objective function \cite{nemhauser1978analysis,5717225}. However, they can provide suboptimal designs, prompting the investigation of other methods. One common alternative is to consider the continuous relaxation of the binary optimization problem, using a regularization term for sparsity. The resulting continuous optimization problem is then solved using gradient-based optimization methods, and the weights in the optimal solution are thresholded to obtain approximately optimal binary designs for the original problem \cite{article}. One can also interpret the binary coefficients as Bernoulli random variables, enabling the use of stochastic optimization techniques, as described in \cite{Attia_2022}. This approach casts the regularized optimality criterion into an objective function in the form of an expectation over a multivariate Bernoulli distribution. Detailed in \cite{eswar2025bayesiandoptimalexperimentaldesigns} is another approach that views sensor placement through the lens of the Column Subset Selection Problem (CSSP) from numerical linear algebra, reformulating the objective function to cast optimal sensor placement as choosing the ``best" columns of a matrix, and we make use of an analogous reformulation for the multifidelity setting.

In contrast to the single-fidelity setting, multifidelity D-optimal sensor placement has seen much less study. The authors in \cite{9152984} suggest a greedy approach that uses column-pivoted QR decomposition for both initialization and refinement, and their experimental results indicate a promising foundation for further study. However, direct comparison with our approach is not straightforward because their formulation assumes that the number of cheap and expensive sensors is fixed in advance. Thus, their optimization problem focuses on determining sensor locations for a fixed allocation, whereas our formulation jointly optimizes both the sensor allocation and placement under a total budget constraint. The authors in \cite{cvxrelax} develop another approach that expresses the inverse of the Fisher information matrix as a sum of rank-one matrices to formulate a convex relaxation of the sensor placement problem, demonstrating strong empirical results. Proposed in \cite{Leskovec} is the CELF algorithm, a combination of greedy approaches shown to offer a constant-factor guarantee for a multifidelity sensor placement problem closely related to the setting we consider. However, beyond these works, the multifidelity case of D-optimal sensor placement seems to be relatively unexplored in the current literature. In this work, we focus on D-optimality due to its computational convenience and widespread use in previous optimal experimental design problems.

\section{Preliminaries}
\label{sec:prelim}
\subsection{Bayesian Inverse Problem}
\label{subsec:state}
We begin by formulating the Bayesian inverse problem that motivates our study of multifidelity sensor placement.
Let the state of a physical system be represented by a vector $\bfu \in \R^N$.
We assume that only a subset of its components can be observed through sensors located at $k < N$ spatial positions and that each measurement is corrupted by additive noise.
The observation model is
\begin{equation}
\bfy = \bS\t \bfu + \bfeps,
\qquad
\bfeps \sim \mathcal{N}(\b{0}, \bSigma_{\text{noise}}),
\label{eq:obs_model}
\end{equation}
where $\bfy \in \R^k$ denotes the vector of noisy measurements,
$\bS \in \{0,1\}^{N \times k}$ is a {selection matrix} that extracts the components of $\bfu$ corresponding to the sensor locations, and
$\bSigma_{\text{noise}}$
represents the sensor noise covariance, which we assume is diagonal. 

If the sensor locations correspond to an index set $\S = \{i_1,i_2,\hdots,i_k\} \subset \{1,\dots,N\}$, then $\bS = \begin{bmatrix}
    \bfe_{i_1} & \bfe_{i_2} & \cdots & \bfe_{i_k}
\end{bmatrix} \in \{0,1\}^{N \times k}$ consists of the columns of the $N\times N$ identity matrix $\bI_N$ indexed by $\S$ where $\bfe_i$ denotes the $i$-th unit vector in $\R^N$. Hence, $\bS\t \bfu$ extracts the components of $\bfu$ at the sensor locations $\S$.

\myparagraph{Dictionary-based representation}
Following the approach proposed in \cite{kakasenko2025bridginggapdeterministicprobabilistic}, we construct a data-driven reduced-order model to mitigate the high computational cost of high-dimensional inference. Specifically, we assume that the state $\bfu \in \R^N$ lies approximately in a low-dimensional subspace spanned by the columns of a dictionary matrix $\bPhi \in \R^{N \times \ell}$, where $\ell \ll N$. The basis $\bPhi$ is obtained from a set of training data (\emph{snapshots}) collected under varying conditions or time instances.
Let $\bY \in \R^{N \times p}$ denote the training data matrix whose columns correspond to these snapshots.
For simplicity of exposition, we assume that the data are mean-centered, i.e., the column mean of $\bY$ is zero\footnote{In the general case, one may apply the same procedure to the mean-centered data $\bY - \bar{\bfu}\mathbf{1}^\top$, reconstruct $\bfu - \bar{\bfu}$, and then add $\bar{\bfu}$ to obtain an estimate of $\bfu$.}\!.
A thin singular value decomposition (SVD) of $\bY$ is written as
\begin{equation*}
\bY = \bU_\bY \bSigma_\bY \bV_\bY\t,
\end{equation*}
and the reduced basis $\bPhi$ is formed by the first $\ell$ left singular vectors, that is, the first $\ell$ columns of $\bU_\bY$, which capture the dominant modes of variability in the data.
This construction yields a data-informed low-dimensional subspace suitable for efficient Bayesian inference.
The reduced-order approximation of the state is then
\begin{align*}
\bfu \approx \bPhi\bfm, \quad \bfm\in \R^\ell,
\end{align*}
where $\bfm$ represents the reduced coordinates of the state in the subspace spanned by $\bPhi$.
Substituting this approximation into the observation model \eqref{eq:obs_model} yields the reduced measurement equation
\begin{equation}
\bfy = \bS\t \bPhi \bfm + \bfeps.
\label{eq:reduced_obs}
\end{equation}
Note here we have ignored the error coming from the truncated SVD, but the formulation can be adjusted appropriately. 
\myparagraph{Bayesian formulation}
We impose a Gaussian prior on the reduced coordinates,
  $\bfm \sim \mathcal{N}(\b{0},\, \bSigma_{\mathrm{pr}})$,
where the prior covariance is chosen as 
\begin{align}
\label{eq:priorCov} \bSigma_{\pr} = \frac{\lambda^2}{p-1}\bSigma_\ell^2,
\end{align}
where $\bSigma_\ell$ is the leading $\ell \times \ell$ block of the empirical snapshot covariance $\bSigma_\bY$. 
The scalar $\lambda$ acts as a regularization parameter similar to Tikhonov regularization, to control balance between the  prior and the data misfit. Methods for selecting this parameter include the discrepancy principle or L-curve analysis; see, e.g., \cite[Ch. 4]{Hansen}. The choice of prior is motivated by the empirical variability of the training data snapshots in the reduced basis, so that the prior covariance reflects the dominant modes of the data through the singular values of the snapshot matrix as described in \cite{kakasenko2025bridginggapdeterministicprobabilistic}.
Given the linear-Gaussian structure of \eqref{eq:reduced_obs}, the posterior distribution of $\bfm$ conditioned on $\bfy$ is also Gaussian,
\begin{align*}
    \bfm \mid \bfy \sim \mathcal{N}(\bfm_{\text{post}},\, \bSigma_{\mathrm{post}}),
\end{align*}
with posterior covariance and  mean given by 
\begin{align*}
\bSigma_{\mathrm{post}}^{-1} &= 
( \bS\t\bPhi)\t  \bSigma_{\mathrm{noise}}^{-1} (\bS\t \bPhi) + \bSigma_{\mathrm{pr}}^{-1},
\\[0.2em]
\bfm_{\mathrm{post}} &=
\bSigma_{\mathrm{post}} ( \bS\t\bPhi)\t \bSigma_{\mathrm{noise}}^{-1} \bfy.
\end{align*}
The posterior mean $\bPhi \bfm_{\mathrm{post}}$ provides a linear estimator of the full state $\bfu$. In the linear-Gaussian setting considered here, the posterior mean coincides with the maximum a posteriori (MAP) estimate. The posterior covariance $\bPhi \bSigma_{\mathrm{post}} \bPhi^\top$ quantifies the uncertainty in this estimate.
Further details on this Bayesian state estimation framework can be found in \cite{kakasenko2025bridginggapdeterministicprobabilistic}.

\subsection{Sensor Placement and the D-optimality Criterion} \label{subsec:single-fidelity}
In practical settings, it is often infeasible to install sensors throughout the entire spatial domain due to constraints in budgetary cost, accessibility, or physical limitations.
Instead, one typically identifies a finite set of feasible or candidate sensor locations and then selects a subset that maximizes the information gained about the system state.
Let \(M \le N\) denote the number of candidate sensor locations, and let 
\(\bS_{\mathrm{cand}} \in \{0,1\}^{N \times M}\) be the corresponding selection matrix, cf.\ \Cref{subsec:state}. Define the matrix $\bPsi = \bS_{\mathrm{cand}}\t \bPhi \in \R^{M \times \ell}$ to be the restriction of the modal basis $\bPhi$ to the candidate sensor locations.
For notational convenience, we represent a sensor configuration by a subset of indices 
\(\S \subseteq \{1,2,\hdots,M\}\), where the indices in $\S$ correspond to column indices in $\bS_{\mathrm{cand}}$. As we will see, using a set is valid for the D-optimality criterion, since the order of selection is irrelevant and duplicate selections are not allowed.

In the Gaussian setting, the posterior covariance of \(\bfu\) fully characterizes uncertainty, 
and its determinant is proportional to the volume of the corresponding uncertainty ellipsoid.
This motivates the Bayesian D-optimality criterion~\cite{alexanderian2023briefnotebayesiandoptimality}, which can be expressed as the maximization of the log-determinant of the  matrix \(\bSigma_\mathrm{post}^{-1}(\S) \bSigma_{\mathrm{pr}}\), leading to
\begin{equation}\label{eq:single_obj_fn}
\Phi_D(\S) 
= \log\det\!\Big(
\bSigma_{\mathrm{pr}}^{1/2} \bSigma_\mathrm{post}^{-1}(\S) \bSigma_{\mathrm{pr}}^{1/2}
\Big), 
\quad \S \subseteq \{1,2,\hdots,M\}, \quad |\S| = k \le M.
\end{equation}
Here, we slightly abuse notation by identifying the index set \(\S\) with its associated selection matrix \(\bS\), which consists of the columns of the identity matrix $\bI_M$ corresponding to the indices in $\S$. The quantities entering the posterior take the form
\begin{align*}
\begin{aligned}
\bSigma_{\mathrm{post}}(\S) 
&= \left((\bS\t\bPsi)\t 
\bSigma_{\mathrm{noise}}^{-1} 
(\bS\t\bPsi) + \bSigma_{\mathrm{pr}}^{-1}\right)^{-1}, \\[0.2em]
\bfm_{\mathrm{post}}
&= \bSigma_{\mathrm{post}} 
(\bS\t\bPsi)\t
\bSigma_{\mathrm{noise}}^{-1} \bfy.
\end{aligned}
\end{align*}

Maximizing $ \Phi_D(\S)$ over all possible selections $\S$ is known to be NP-hard~\cite{eswar2025bayesiandoptimalexperimentaldesigns}.
This motivates the use of approximate methods such as greedy and relaxation-based approaches.  
A standard greedy strategy, originating in the theory of submodular set functions~\cite{nemhauser1978analysis,5717225},  
has been successfully applied in Bayesian experimental design and sensor placement problems~\cite{JMLR:v9:krause08a,fluid1,kakasenko2025bridginggapdeterministicprobabilistic}.  
In the single-fidelity setting, the greedy algorithm proceeds as follows:  
starting with $\S = \emptyset$, it iteratively adds the sensor location that yields the largest improvement in $\Phi_D(\S)$,
\begin{align*}
i^* \in  
\argmax_{i \in \S_{\mathrm{cand}} \setminus \S}
\left\{
\Phi_D(\S \cup \{i\}) - \Phi_D(\S)
\right\},
\end{align*}
then updates $\S \gets \S \cup \{i^*\}$ until $|\S| = k$.  
The resulting set $\S$ represents an approximately D-optimal sensor configuration. In particular, since $\Phi_D$ is a monotone, submodular, and normalized set function, the objective value of the sensor configuration produced by this algorithm is guaranteed to be within a constant factor $1-e^{-1}$ of the optimal value \cite{nemhauser1978analysis}.
In \Cref{sec:mfgreedy}, we extend this approach to the multifidelity setting.

\subsection{Linear Algebra Background}
We briefly review pertinent matrix identities used in our analysis.

\begin{lemma}[Sherman-Morrison Formula {\cite[Sec.~2.1.4]{golub2013matrix}}]
\label{eq:sherman}
    Let $\bA\in \R ^{ n \times n}$ be an invertible matrix and let $\bfu, \bfv \in \R^n$. 
    The matrix $\bA + \bfu\bfv\t$ is invertible if and only if $1 + \bfv\t\bA^{-1}\bfu \neq 0$.
    In this case, the inverse is given by the formula 
    \begin{align*}
        \left(\bA + \bfu \bfv\t\right)^{-1} = \bA^{-1} - \frac{\bA^{-1}\bfu \bfv\t\bA^{-1}}{1 + \bfv\t\bA^{-1}\bfu}.
    \end{align*}
\end{lemma}

\begin{lemma}[Matrix Determinant Lemma {\cite[Example 1.3.24.]{HornJohnson2012}}]
\label{lemma:matrix-det-lemma}
    Let $\bA\in \R ^{ n \times n}$ be an invertible matrix and let $\bfu, \bfv \in \R^n$. Then 
     \begin{align}\label{eq:Matrix_Determinant}
    \det\left(\bA + \bfu \bfv\t\right) = \left(1 + \bfv\t\bA^{-1} \bfu\right)\det\bA.
    \end{align}  
\end{lemma}

\begin{lemma}[Minkowski Determinant Theorem {\cite[Sec.~4.1.8]{Marcus1965ASO}}]
\label{lemma:minkowski}
    Let $\bA,\bB \in \R^{n \times n}$ be symmetric positive semidefinite (SPSD) matrices. Then
    \[\left(\det(\bA+\bB)\right)^{\frac1n} \geq (\det\bA)^{\frac1n} + (\det\bB)^{\frac1n}.\]
\end{lemma}

\begin{corollary}
\label{cor:minkowski}
Let $\bA,\bB \in \R^{n \times n}$ be SPSD matrices. Then
\[\det(\bA+\bB) \geq \det\bA + \det\bB.\]
\end{corollary}
\begin{proof}
By \Cref{lemma:minkowski},
\[\left(\det(\bA+\bB)\right)^{\frac1n} \geq (\det\bA)^{\frac1n} + (\det\bB)^{\frac1n}.\]
Since the determinant of an SPSD  matrix is nonnegative and the map $x \mapsto x^n$ is increasing on $\R_{\geq0}$, from binomial expansion we obtain 
\begin{align*}
\det(\bA+\bB) &\geq \left((\det\bA)^{\frac1n} + (\det\bB)^{\frac1n}\right)^n \\&= \det\bA + \det\bB + \sum_{k=1}^{n-1} \binom{n}{k} (\det\bA)^{\frac kn} (\det\bB)^{\frac{n-k}n} \\&\geq \det \bA + \det \bB.
\end{align*}
\end{proof}

\section{Multifidelity Sensor Placement}

\label{sec:mfgreedy}
In this section, we introduce the multifidelity sensor placement problem, in which sensors of varying quality and cost must be selected under a global budget constraint. We formalize the design criterion and establish key properties of the resulting D-optimality objective. These properties motivate the use of greedy methods, which we then develop and analyze, including an efficient implementation based on Sherman-Morrison updates and a discussion of their theoretical performance. 
\subsection{Problem Setup and Design Criterion}
We now extend our formulation to the multifidelity sensor placement problem. For clarity of exposition, we consider two types of sensors, cheap and expensive. Cheap sensors are less accurate but less costly, while expensive sensors provide higher-quality measurements at a higher cost.

As before, we consider $M \le N$ different candidate sensor locations. At each candidate location, one may place either a cheap sensor, an expensive sensor, or no sensor at all, with the restriction that no two sensors of any type occupy the same location. Each sensor type is characterized by its associated cost and the standard deviation in the noise of its measurements.
Let $\sigmaC$ and $\sigmaE$
be the measurement noise standard deviations for cheap and expensive sensors, respectively, 
and let $\cC$ and $\cE$ be
the corresponding costs. These costs may represent financial, installation, or operational constraints.
Throughout this article we assume
\[
    \sigmaC > \sigmaE, 
    \qquad 
    \cC < \cE,
\]
reflecting the trade-off between accuracy and cost.

In the single-fidelity case, since there is only one type of sensor, specifying a budget is tantamount to specifying the number of sensors. However, in the multifidelity case, we assume that there is a global budget $\bud$, and deciding the apportionment of the cheap and expensive sensors is a part of the problem.
Let $\Sc$, $\Se  \subseteq \{1, 2, \dots , M\}$ be index sets corresponding to the locations of cheap and expensive sensors, respectively, and let $\kc = |\Sc|$ and $\ke = |\Se|$. By the assumptions made thus far, 
\begin{align*}
    \Sc\cap\Se = \emptyset.
\end{align*}
Let $\bSc \in \{0,1\}^{M \times \kc}$ and $\bSe \in \{0,1\}^{M \times \ke}$
denote the selection matrices corresponding to cheap and expensive sensors $\Sc$ and $\Se $ respectively, and let $\bS = \begin{bmatrix}
    \bSc & \bSe
\end{bmatrix}$. The measurement model is 
\[ \bfy = \bS\t\bPsi \bfm + \bfeps,  \qquad \bfeps \sim \mathcal{N}(\b{0},\bSigma_{\rm noise}), \]
where the covariance matrix of the measurement noise is given by
\begin{equation*}
\bSigma_{\text{noise}} = \begin{bmatrix}
    \sigmaC^2 \bI_{\kc} & \\
    & \sigmaE^2 \bI_{\ke}
\end{bmatrix}.
\end{equation*}
With the prior distribution $\bfm \sim \mathcal{N}(\b{0},\bSigma_\pr)$, the posterior distribution takes the form $\bfm | \bfy \sim \mathcal{N}(\bfm_{\rm post}, \bSigma_{\rm post})$, where
\begin{align*}
\begin{aligned}
\bSigma_{\rm post} &= \left(\left(\bPsi\t\bS\right) \bSigma_{\text{noise}}^{-1} \left(\bS\t\bPsi\right) + \bSigma_{\pr}^{-1}\right)^{-1}, \\[0.2em]
\bfm_{\rm post} &= \bSigma_{\rm post}(\bPsi\t\bS)\bSigma_{\rm noise}^{-1}\bfy.
\end{aligned}
\end{align*}
 In what follows, it will be convenient to make explicit the dependence of the posterior covariance matrix on the selection list $S = [\Sc,\Se]$. Thus, we define
\begin{equation*}
\bSigma_\mathrm{post}^{-1}(S) = \left(\bPsi\t\bS\right) \bSigma_{\text{noise}}^{-1} \left(\bS\t\bPsi\right) + \bSigma_{\pr}^{-1},
\end{equation*}
where $\bS = \begin{bmatrix}
    \bSc & \bSe
\end{bmatrix}$ is the selection matrix for both cheap and expensive sensors corresponding to $S$.

\myparagraph{Problem statement} We are now ready to give the formal problem statement. As in the single-fidelity case, we aim to maximize the D-optimality criterion
\begin{equation}\label{multi_obj_fn}
 \Phi_D(S) := \log\det\left(\bSigma_{\mathrm{pr}}^{\frac12} \bSigma_\mathrm{post}^{-1}(S) \bSigma_{\mathrm{pr}}^{\frac12}\right)
\end{equation}
over all lists of non-overlapping selection sets $S = [\Sc,\Se]$, subject to the constraint that $\cC |\Sc| + \cE |\Se| \leq \bud$. The difference compared to the single-fidelity case is the budgetary  constraint and the dependence of the posterior covariance on the choice of sensors. 

This cost-constrained sensor placement problem can be viewed as a variant of the knapsack problem \cite{Korte2008}, where each candidate sensor has an associated cost and contributes to the Bayesian D-optimality objective $\Phi_{D}$. Exact solution of such combinatorial problems becomes intractable for large numbers of candidates $M$, motivating the use of greedy and heuristic strategies to obtain approximate solutions. Following \cite{eswar2025bayesiandoptimalexperimentaldesigns}, we recast the problem as a column subset selection problem. With this formulation, the D-optimality criterion can be equivalently expressed as follows.

\begin{proposition}\label{prop:I+AAT}
For each fidelity type $j \in \{\ch,\exp\}$, define
\[
\bA_j := \sigma_j^{-1}\,
\bSigma_{\mathrm{pr}}^{1/2}\,
\bPsi^\top.
\]
For a selection $S = [\Sc,\Se]$, we can write $\Phi_D(S) = \log\det\bigl(\bB(S)\bigr)$, where
\[
\bB(S)
:= \bI +
(\bAc \bSc)(\bAc \bSc)^\top
+
(\bAe \bSe)(\bAe \bSe)^\top.
\]

\end{proposition}
\begin{proof} The proof follows from 
\begin{align*}
\bSigma^{\frac 12}_{\mathrm{pr}}\bSigma_\mathrm{post}^{-1}(S) \bSigma^{\frac 12}_{\mathrm{pr}} = \bSigma^{\frac 12}_{\mathrm{pr}}\left(\bPsi\t\bS \bSigma_{\text{noise}}^{-1} \bS\t\bPsi + \bSigma_{\pr}^{-1}\right)\bSigma^{\frac 12}_{\mathrm{pr}}
= \bI + \bSigma_{\mathrm{pr}}^{\frac12} \bPsi\t\bS \bSigma_{\text{noise}}^{-1} \bS\t\bPsi\bSigma_{\mathrm{pr}}^{\frac12}.
\end{align*}
Expanding the matrix on the right,
\begin{align*}\bSigma_{\mathrm{pr}}^{\frac12} \bPsi\t\bS \bSigma_{\text{noise}}^{-1} \bS\t\bPsi\bSigma_{\mathrm{pr}}^{\frac12} &= \bSigma_{\mathrm{pr}}^{\frac12} \bPsi\t\begin{bmatrix} \bSc & \bSe \end{bmatrix} \begin{bmatrix} \sigmaC^{-2} \bI & \\ & \sigmaE^{-2}\bI \end{bmatrix} \begin{bmatrix} \bSc\t\\ \bSe\t\end{bmatrix} \bPsi\bSigma_{\mathrm{pr}}^{\frac12}
\\&= \begin{bmatrix} \bAc \bSc & \bAe \bSe \end{bmatrix} \begin{bmatrix}
    \bSc\t\bAc\t\\ \bSe\t\bAe\t
\end{bmatrix}
\\&= \left(\bAc \bSc\right)\left(\bAc \bSc\right)\t+ \left(\bAe \bSe\right)\left(\bAe \bSe\right)\t.
\end{align*}
Plugging this expansion into the definition of $\Phi_D$ yields 
the desired result.
\end{proof}
As noted above, this result allows us to view optimal sensor placement through the lens of column subset selection, as we aim to choose those columns $\{\bfa_i\}_{i=1}^k$ from $\bAc$ and $\bAe$ that maximize $\log\det\left(\bI + \sum_{i=1}^k \bfa_i \bfa_i\t\right)$ subject to a budget constraint.
Moreover, expressing the objective function in this way enables computationally efficient implementations of approaches that add one sensor at a time, such as greedy approaches, via rank-one updates. Namely, when adding a sensor of fidelity $j\in\{\ch,\exp\}$ at location $i \in \{1,\dots,M\}$, one can add the rank-one matrix $\left([\bA_j]_{:i}\right)\left([\bA_j]_{:i}\right)\t$ to the matrix inside the log-determinant in \Cref{prop:I+AAT}, and the resulting expression is equal to $\Phi_D\left(S^{(j,i)}\right)$, where the notation $S^{(j,i)}$ denotes
\begin{equation*}
S^{(j,i)} =
\begin{cases}
\big[\Sc \cup \{i\}, \,\Se \big], & \text{if } j = \ch, \\[4pt]
\big[\Sc, \,\Se \cup \{i\}\big], & \text{if } j = \exp.
\end{cases}
\end{equation*}
\subsection{Properties of the objective function}
\label{subsec:properties}
The single-fidelity objective in \eqref{eq:single_obj_fn} is known to be monotone and submodular with respect to the set of selected sensors \cite{Robertazzi}. That is, adding a sensor never decreases the objective (monotonicity), and the marginal gain from adding a sensor decreases as the selected set grows (submodularity). These properties provide theoretical justification for greedy selection, which is both empirically and analytically effective for maximizing monotone, submodular set functions such as \eqref{eq:single_obj_fn}  (see, e.g.,~\cite{kandasamy,Khuller,Minoux,nemhauser1978analysis,SVIRIDENKO200441,5717225}). 
Thus, we first extend some of these properties to multifidelity sensor placement, where the objective function is $\Phi_D$ in \eqref{multi_obj_fn}, which partially justifies the use of greedy algorithms in the multifidelity case.
\begin{proposition}[Monotonicity and Submodularity]
\label{prop:monotone}
\label{add_sensor}
Let $S = [\Sc,\Se]$ and ${\widetilde  S} = [\tSc, \tSe  ]$  be two selection lists such that
\[
\Sc\cap \Se = \tSc \cap \tSe = \emptyset,
\qquad
S_j \subseteq \widetilde{S}_j
\quad \text{for } j \in \{\ch,\exp\}.
\]
Then the D--optimality criterion $\Phi_D$ satisfies the following properties:

\begin{enumerate}
    \item \textbf{Monotonicity.} $\Phi_D(S) \le \Phi_D(\widetilde{S})$.
    \item \textbf{Submodularity.} Let $i \notin \tSc \cup \tSe$ and let $j \in \{\ch,\exp\}$. Then
\[
\Phi_D\!\left(S^{(j,i)}\right) - \Phi_D(S)
\;\ge\;
\Phi_D\!\left(\widetilde{S}^{(j,i)}\right) - \Phi_D(\widetilde{S}).
\]
\end{enumerate}
\end{proposition}
\begin{proof}
If $S_j = {\widetilde  S}_j$ for all $j \in \{\ch,\exp\}$, both results are immediate. Otherwise, assume without loss of generality that there exists $j \in \{\ch, \exp\}$ and $i \in \widetilde S_{j}$ such that ${\widetilde S} = S^{(j,i)}$. To obtain the general case, one can apply the result for the assumed case repeatedly while adding one sensor at a time to $S$.

Let $\bSc$ and $\bSe$ denote the cheap and expensive selection matrices for $\Sc$ and $\Se $, respectively. 
Let $\bfa = [\bA_{j}]_{:i}$ be the $i$-th column of $\bA_{j}$, and let $\bfh = [\bA_{j'}]_{:i'}$ be the $i'$-th column of $\bA_{j'}$. 
Let $\bB(S)$ be as in \Cref{prop:I+AAT}, and for brevity write $\bB$ instead of $\bB(S)$. Note that $\bB$ is symmetric positive definite (SPD). We tackle each property:

\textbf{1. Monotonicity}: Since $\bB$ and $\bfa\bfa\t$ are both SPSD, \Cref{cor:minkowski} applies, which gives
\begin{align*}
 \det\left(\bB + \bfa\bfa\t\right) \geq \det\left(\bB\right) + \det\left(\bfa\bfa\t\right) \geq \det\left(\bB\right) 
\end{align*}
as the determinant of a positive semidefinite matrix is nonnegative. Taking logarithms on both sides gives the result $\Phi_D({\widetilde  S}) \geq \Phi_D(S)$.

\textbf{2. Submodularity}: By \Cref{prop:I+AAT} and \Cref{lemma:matrix-det-lemma},
\begin{equation}\label{eqn:inter}
\begin{aligned}
\Phi_D(S^{(j',i')}) - \Phi_D(S) &= \log\det\left(\bB + \bfh\bfh\t\right) - \log\det(\bB)  = \log(1 + \bfh\t \bB^{-1} \bfh), \\ 
\Phi_D({\widetilde S}^{(j',i')}) - \Phi_D({\widetilde S}) &= \log\det\left(\bB + \bfa\bfa\t + \bfh\bfh\t\right) - \log\det\left(\bB + \bfa\bfa\t\right) 
\\&= \log(1 + \bfh\t (\bB+\bfa\bfa\t)^{-1} \bfh).
\end{aligned}
\end{equation}
Applying the Sherman-Morrison formula \eqref{eq:sherman},
\begin{align*}
\bfh\t (\bB+\bfa\bfa\t)^{-1} \bfh = & \>   \bfh\t \left(\bB^{-1} - \frac{\bB^{-1} \bfa \bfa\t \bB^{-1}}{1+\bfa\t\bB^{-1}\bfa}\right) \bfh \\
= & \> \bfh\t\bB^{-1}\bfh - \frac{(\bfh\t \bB^{-1} \bfa)^2}{1+\bfa\t\bB^{-1}\bfa} \le \bfh\t\bB^{-1}\bfh.
\end{align*}
We have used that  $\bB$ is SPD.  Substituting this inequality into~\eqref{eqn:inter} completes the proof.
\end{proof}
Before introducing the greedy algorithm, we state a result that will be used later in \Cref{sec:iterative}. In contrast to the previous results, this one is specific to the multifidelity setting and has no counterpart in the single-fidelity case, as it directly compares sensors of different fidelities. Informally, the result shows that, for a fixed sensor location, placing a high-fidelity sensor is at least as informative as placing a low-fidelity sensor.
\begin{proposition}[Low-noise sensors are more informative]
\label{lower_variance_sensor}
Let $S = [\Sc,\Se]$ be a selection list with $\Sc \cap\Se = \emptyset$. Let $i \in \{1,2,\hdots,M\} \setminus \left(\Sc \cup\Se\right)$ be a sensor location, previously not selected. Then, $\Phi_D\left(S^{(\ch,i)}\right) \leq \Phi_D\left(S^{(\exp,i)}\right)$.
\end{proposition}
\begin{proof}

Consider the same notation as in the proof of \Cref{prop:monotone}. For each $j \in \{\ch,\exp\}$ we have 
\begin{align*}
\Phi_D\left(S^{(j,i)}\right) 
&= \log\det\left(\bB + \sigma_j^{-2}\bfa\bfa\t\right).
\end{align*}

Since  $\sigmaE \le \sigmaC$, we can apply \Cref{cor:minkowski} to get 
\[\Phi_D\left(S^{(\ch,i)}\right) = \log\det(\bB+\sigmaC^{-2} \bfa\bfa\t) \leq \log \det(\bB+\sigmaE^{-2} \bfa\bfa\t) = \Phi_D\left(S^{(\exp,i)}\right)\]
which is the desired result.
\end{proof}
\subsection{Greedy Algorithm}
\label{subsec:mfgreedy}
A natural extension of the single-fidelity greedy algorithm, following the formula in \cite[Eq.~(3)]{SVIRIDENKO200441}, 
addresses the cost-constrained multifidelity setting by maximizing the marginal improvement in the objective function, i.e., the improvement per unit cost.

This generalized greedy approach is summarized as follows. Initialize the sets $\Sc = \emptyset$ and $\Se  = \emptyset$. We maintain a variable $b_\curr$ to represent the current available budget, which is initialized as the total budget $\bud$. While there is a remaining budget (i.e., $b_\curr > 0$), define the sets  $J = \{ j \in \{\ch,\exp\}| c_j  \le b_\curr\}$, which represents the set of fidelities for which there is sufficient remaining budget, and $I \gets \{1,\dots,M\} \setminus \left(\Sc \cup\Se\right)$, which represents the remaining sensor locations. 

At each iteration, we choose the index pair $(j^*, i^*)$ that maximizes the marginal improvement:
\[(j^*, i^*) \in \argmax_{(j,i)\in J \times I} ~\frac{1}{c_j}\left({\Phi_D(S^{(j,i)}) - \Phi_D\left(S\right)}\right).\]

We then update the selected set by $S_{j^*} \gets S_{j^*} \cup \{i^*\}$ and reduce the remaining budget as $b_{\curr} \leftarrow b_{\curr} - c_{j^*}$.  
The process continues until no additional sensor can be afforded, namely $b_{\curr} < \cC$. A full pseudocode description of this multifidelity greedy selection procedure is provided in \Cref{mfgreedy}. 

\myparagraph{Remark}
Although we focus on the two-class setting throughout this work, the greedy framework extends naturally to an arbitrary number of sensor classes by replacing $\{\ch,\exp\}$ with a general index set of sensor fidelities and maximizing over all admissible fidelity-location pairs at each iteration. 

\begin{algorithm}[!ht]
\caption{Greedy Algorithm for Multifidelity Sensor Selection}
\label{mfgreedy}
\begin{algorithmic}[1]
\item[\textbf{Input:}]
Sensor costs $\cC < \cE$, measurement noise standard deviations $\sigmaC > \sigmaE$, budget $\bud > 0$, prior covariance matrix $\bSigma_{\mathrm{pr}}$, modal basis $\bPhi \in \R^{N \times \ell}$ 

\item[\textbf{Output:}] Selected sensor index sets $S = [\Sc,\Se]$
        \STATE Initialize $\Sc \gets \emptyset$, $\Se  \gets \emptyset$, $b_\curr \gets \bud$, $t \gets 1$
    \WHILE{$b_\curr \ge \cC$}
    \STATE $J \gets \{j \in \{\ch, \exp\} \mid c_j \le b_\curr\}$\hfill  \codecomment{Admissible fidelity types}
           \STATE $I \gets \{1,\ldots,M\} \setminus (\Sc \cup\Se)$ \hfill\codecomment{Available locations}
            \STATE\label{lin:greedydom} $(j_t^*, i_t^*) \gets \underset{(j,i) \in J \times I}{\arg\max} \{ {c^{-1}_j}({\Phi_D \left(S^{(j,i)}\right) - \Phi_D(S)} )\}$  \hfill \codecomment{Information gain per cost}
            \STATE ${S}_{j_t^*} \gets S_{j_t^*} \cup \{i_t^*\}$ and $S \gets [\Sc,\Se]$
            \STATE $b_\curr \gets b_\curr - c_{j^*_t}$ and $t \gets t + 1$
    \ENDWHILE
\end{algorithmic}
\end{algorithm}

\myparagraph{Computational Cost} \label{subsubsec:naive} 
We now analyze the computational cost of \Cref{mfgreedy}. The algorithm iteratively adds one sensor at each step until a total of $k$ sensors are selected. At iteration $1 \leq t \leq k$, the dominant operation corresponds to Line~\ref{lin:greedydom}, which seeks the pair $(j,i)$ that maximizes the marginal improvement in the objective $\Phi_D$ 
across each sensor fidelity and location, which requires evaluating $\Phi_D(S^{(j,i)})$ once for each fidelity $j$ for which there is enough remaining budget and each location $i$ at which a sensor has not already been placed. There are at most two such fidelities, and at the $t$-th iteration there are $M-t+1$ such locations, so the number of evaluations of $\Phi_D$ at the $t$-th iteration is $2(M-t+1)$. More generally, if $f$ sensor fidelity classes are considered, the number of evaluations at the $t$-th iteration scales as $f(M-t+1)$. 
The time complexity of the algorithm depends on how the evaluations are done. In the na\"ive approach, the matrix $\bI + \left(\bAc\bSc\right)\left(\bAc\bSc\right)\t+ \left(\bAe\bSe\right)\left(\bAe\bSe\right)\t$ is formed explicitly at each evaluation, and then its log-determinant is computed. Recall that $\ell$ denotes the number of reduced basis vectors in the matrix $\bPsi$. At the $t$-th iteration, forming this matrix requires $\bigO{t\ell^2}$ f{}lops, and computing its log-determinant requires $\bigO{\ell^3}$ f{}lops. Thus, the total cost is
\begin{align*}
\bigO{\sum_{t=1}^k 2(M-t+1)(t\ell^2 + \ell^3)} 
&= \bigO{\ell^2\sum_{t=1}^k 2t(M-t+1) + \ell^3 \sum_{t=1}^k 2(M-t+1)}.
\end{align*}
Expanding the left sum,
\begin{align*}
\ell^2\sum_{t=1}^k 2t(M-t+1) &= 2\ell^2\left((M+1)\sum_{t=1}^k t - \sum_{t=1}^k t^2\right)
\\&= \ell^2 \left((M+1)(k+1)k - \frac{k(k+1)(2k+1)}{3}\right)
\\&= \bigO{k^2M\ell^2},
\end{align*}
since $k \leq M$. Similarly, for the right sum, we have
\begin{align*}
\ell^3 \sum_{t=1}^k 2(M-t+1) = \ell^3\left(2k(M+1) - k(k+1)\right) = \ell^3(2kM-k^2+k) = \bigO{kM\ell^3}.
\end{align*}
Combining the two sums, the total cost is $\bigO{kM\ell^2(k+ \ell)}$ f{}lops. We now discuss a more efficient implementation. 
\subsection{Greedy Algorithm with Sherman--Morrison updates}
\label{subsec:sherman}
A direct evaluation of the D-optimality objective~$\Phi_D$ can be prohibitively expensive, particularly in large-scale settings (see \Cref{subsubsec:naive}). To reduce the computational burden, we exploit the structure of rank-one updates arising in the greedy selection process. In particular, we combine the matrix determinant lemma (\ref{eq:Matrix_Determinant}) with the Sherman--Morrison formula~\eqref{eq:sherman} to enable efficient updates of both the objective value and the marginal gains.
\myparagraph{Main idea} Suppose at iteration $1 \leq t \leq k$  we have selected sensors $S = [\Sc,\Se]$, with corresponding selection matrix
$\bS = \begin{bmatrix}
    \bSc & \bSe
\end{bmatrix}$. As in \Cref{prop:I+AAT}, define the matrix
\begin{align*}
\bB_{t-1} = \bI + \left(\bAc\bSc\right)\left(\bAc\bSc\right)\t+ \left(\bAe\bSe\right)\left(\bAe\bSe\right)\t, \quad \text{with} \quad \bB_0 = \bI.
\end{align*}
By construction, the D-optimality objective can be written as $\Phi_D(S)= \log\det(\bB_{t-1})$.
Consider augmenting the current design by adding a sensor of fidelity $j \in \{\ch,\exp\}$ at location $i \in I$, then by~\Cref{prop:I+AAT}, 
 the corresponding objective value is
\begin{align} \label{eq:Phi_candidate}
\Phi_D\!\left(S^{(j,i)}\right)
= \log\det\!\left(\bB_{t-1} + [\bA_j]_{:i}[\bA_j]_{:i}\t \right),
\end{align}
where $[\bA_j]_{:i}$ denotes the $i$-th column of $\bA_j$ for $j \in J$ and $i \in I$.  
The argument of the log-determinant in~\eqref{eq:Phi_candidate} is a rank-one update of the matrix $\bB_{t-1}$.
Therefore, by the matrix determinant lemma (\ref{eq:Matrix_Determinant}), the marginal gain associated with adding this sensor is given by
\begin{align*}
\label{SM1}
{\Phi_D(S^{(j,i)}) - \Phi_D(S)} = & \> {\log\det(\bB_{t-1}+[\bA_j]_{:i}[\bA_j]_{:i}\t) - \log\det(\bB_{t-1})} \\
=& \> {\log(1+[\bA_j]_{:i}\t\bB_{t-1}^{-1} [\bA_j]_{:i})}. 
\end{align*}

Let $\bfu_t = [\bA_{j_t^\ast}]_{:i_t^\ast}$ denote the column corresponding to the selected fidelity $j_t^\ast$ and location $i_t^\ast$ at iteration $t$, as determined by \Cref{mfgreedy},  with  $\bfu_0 = \b{0}$.
The matrix $\bB_t$ then admits the update
\[
\bB_t = \bB_{t-1} + \bfu_t \bfu_t\t.
\]
Since this is a symmetric rank-one update, $\bB_t$ remains SPD for all $t$.
Rather than forming or storing $\bB_t^{-1}$; explicitly, we maintain the auxiliary matrices\footnote{In fact, we actually only need to store $\bB_t^{-1} [\bA_j]_{:i}$ for fidelities $j$ for which there is enough remaining budget and locations $i$ where a sensor has not already been placed, but for clarity of exposition we store the entire matrix $\bB_t^{-1} \bA_j$.} $\bD_j^{(t)} \equiv \bB_t^{-1} \bA_j$ for $j \in \{\ch, \exp\}$ and $t \ge 0$. Since $\bB_0 = \bI$, we initialize $\bD_j^{(0)} \gets \bA_j$ for each fidelity $j \in \{\ch, \exp\}$. 
To update $\bD_j^{(t)}$ efficiently, we apply the Sherman--Morrison formula.
For each $j \in \{\ch,\exp\}$,
\begin{align}\label{eq:InPlaceUpdate}
 \begin{aligned} \bD_j^{(t)} = & \>  \bB_{t}^{-1}\bA_j \\ = & \> \left(\bB_{t-1}^{-1} - \frac{(\bB_{t-1}^{-1}\bfu_{t})(\bB_{t-1}^{-1}\bfu_{t})\t}{1 + \bfu_{t}\t\bB_{t-1}^{-1}\bfu_{t} } \right) \bA_j
  = \> \left(\bI - \frac{1}{1 + \bfu_{t}\t\bfv_{t}}\bfv_{t} \bfu_{t}\t\right)\bD_j^{(t-1)},  
 \end{aligned}
 \end{align}
where $\bfv_{t} = \bB_{t-1}^{-1}\bfu_{t} = [\bD^{(t-1)}_{j_{t}^*}]_{i_{t}^*}$ for $t\ge 1$. Note that the last line holds as $\bB_{t-1}^{-1}$ is a symmetric matrix.
That is, as $\bD_{j_{t}^*}^{(t-1)}$ is a matrix corresponding to the optimal fidelity $j_t^*$, we take $\bfv_{t}$ to be a column of this matrix corresponding to the optimal location $i_{t}^*$. This gives an efficient way to update the matrices $\bD_j^{(t)}$, without having to form and store $\bB_t$ or $\bB_t^{-1}$ at all.  
\myparagraph{Summary}
At iteration $t \ge 1$, the algorithm selects a new sensor by maximizing the marginal gain in the D-optimality criterion, normalized by sensor cost.
Specifically, among all fidelities $j$ for which sufficient budget remains and all locations $i \in I$ at which a sensor has not yet been placed, the greedy selection in Line~\ref{lin:greedydom} is given by
\begin{equation}
\label{eq:mfgreedySM-objective}
    (j_t^*,i_t^*) = \argmax_{(j,i) \in J\times  I} \left\{{c_j^{-1}\ }{\log\left(1+[\bA_j]_{:i}\t[\bD^{(t-1)}_{j}]_{:i}\right)}\right\}.
\end{equation}
The selected sensor of fidelity $j_t^\ast$ is placed at location $i_t^\ast$, after which the remaining budget is updated according to
$\bud \gets \bud - c_{j_t^\ast}$.
The auxiliary matrices $\bD_j^{(t)}$ are then updated for each $j \in \{\ch,\exp\}$ using the Sherman--Morrison update derived above.
This procedure is repeated until the remaining budget is insufficient to place any additional sensor.
A complete pseudocode description of the algorithm is provided in \Cref{mfgreedySM}.
\begin{algorithm}[!ht]
\newpage

\begin{algorithmic}[1]
\caption{Multifidelity Greedy Algorithm with Sherman-Morrison Updates}
\label{mfgreedySM} 
\item[\textbf{Input:}] Sensor costs $\cC < \cE$, measurement noise standard deviations $\sigmaC > \sigmaE$, budget $\bud > 0$, prior covariance matrix $\bSigma_{\mathrm{pr}}$, modal basis $\bPhi \in \R^{N \times \ell}$ \\
\item[\textbf{Output:}] List of index sets [$\Sc,\Se$], the near-optimal cheap and expensive sensor indices, respectively 
    \STATE $\bAc \gets \sigmaC^{-1} {\bSigma_{\pr}^{1/2}\bPsi^{T}}{}$, $\bAe \gets \sigmaE^{-1}{\bSigma_{\pr}^{1/2}\bPsi^{T}}$
    \STATE $I\gets \{1,...,M\}$,  $J\gets \{\ch,\exp \}$, $\Sc \gets \emptyset$ and $\Se  \gets \emptyset$
    \STATE Initialize $\bD_j^{(0)} \gets \bA_j$ for each $j \in J$
    \STATE $b_{\curr} \gets \bud$ and $t \gets 1$
        \WHILE{$b_\curr \ge \cC$}
            \STATE $J \gets \{j  \in J \mid c_j \leq b_\curr\}$  \hfill \codecomment{Admissible fidelity types}
            \STATE $(j_t^*,i_t^*) \gets \underset{(j,i) \in J \times I}{\arg\max} \left\{{c_j^{-1}\ }{\log\left(1+[\bA_j]_{:i}\t[\bD^{(t-1)}_{j}]_{:i}\right)}\right\}$
            \STATE ${S}_{j^*_{t}} \gets S_{j^*_{t}} \cup \{i^*_{t}\}$
             \STATE $I\gets I \setminus \{i^*_{t}\}$
             \hfill 
           \codecomment{Set of locations where a sensor has not been placed}
            \STATE $\bfu_t \gets [\bA_{j_{t}^*}]_{: i_{t}^*}$
            \STATE $\bfv_{t} \gets [\bD^{(t-1)}_{j_{t}^*}]_{: i_{t}^*}$
            \FOR{$j \in J$}
            \STATE $\bD_j^{(t)} \gets \bD_j^{(t-1)} - \frac{1}{1 + \bfu_{t}\t\bfv_{t}}\bfv_{t} \bfu_{t}\t\bD_j^{(t-1)}$
            \ENDFOR
            \STATE $b_{\curr} \gets b_{\curr} - c_{j_t^*}$ and  $t \gets t+1$
    \ENDWHILE
     \RETURN $[\Sc,\Se]$
\end{algorithmic}
\end{algorithm}
\myparagraph{Computational cost}\label{para:greedy-sm-cost} As in \Cref{subsubsec:naive}, this improved greedy procedure needs $\bigO{kM}$ evaluations of marginal gains.
However, in the present formulation, each evaluation reduces to computing the objective in~\eqref{eq:mfgreedySM-objective}, which requires $\bigO{\ell}$ f{}lops.
This reduction in the cost of evaluating the marginal gain is accompanied by the expense of updating the auxiliary matrices $\bD_j^{(t)}$ via the Sherman--Morrison formula~\eqref{eq:InPlaceUpdate}.
Since $\bfu_t \in \mathbb{R}^\ell$ and $\bD_j^{(t-1)} \in \mathbb{R}^{\ell \times M}$, each such update requires $\bigO{M\ell}$ flops per iteration, as $\bfu_t \in \R^\ell$ and $\bD_j^{(t-1)} \in \R^{\ell \times M}$.
Consequently, the total computational cost over $k$ iterations is
\[
\bigO{\ell(kM) + k(M\ell)} = \bigO{kM\ell} ~\text{f{}lops.}
\]
This is much more efficient compared to the na\"ive greedy implementation (\Cref{mfgreedy}). 
\subsection{Theoretical Performance}
\label{subsec:performance}
An important question is whether the proposed greedy strategy admits a worst-case performance guarantee comparable to those available in the single-fidelity setting (cf.~\Cref{subsec:single-fidelity}). In particular, one might ask whether there exists a constant $\alpha \in (0,1)$ such that, for all admissible problem instances, the D-optimality value achieved by the greedy solution is at least an $\alpha$-fraction of the optimal value.

The result below shows that such a uniform constant-factor approximation guarantee does not hold in general. Using a counterexample, adapted from a related argument in \cite{Leskovec}, we demonstrate that the performance of the greedy algorithm can be arbitrarily poor relative to the optimum.
\begin{proposition}[Greedy is not a constant factor approximation]
\label{prop:counterexample}
Given $\epsII \in (0,1)$ and the matrices $\bSigma_\pr$ and $\bPhi$, there exists a set of candidate locations $\S_\mathrm{cand} \subseteq \{1,2,\hdots,N\}$, measurement noise standard deviations $\sigmaC > \sigmaE > 0$, costs $\cE > \cC > 0$, and a budget $\bud  > 0$ such that, letting $S_{\mathrm{opt}}$ and $S_{\mathrm{greedy}}$ be an optimal solution and the solution produced by \Cref{mfgreedySM}, respectively, we have
\[\Phi_D(S_{\mathrm{greedy}}) \leq \epsII \Phi_D(S_{\mathrm{opt}}).\]
\end{proposition}
\begin{proof}
Let $\S_\mathrm{cand} \subseteq \{1,2,\hdots,N\}$ with $|\S_\mathrm{cand}|=1$, i.e., there is only a single candidate sensor location. Then $\bPsi = \bS_\mathrm{cand}\t \bPhi$ has only one row, which implies $\bfa := \bSigma_\pr^{1/2} \bPsi\t$ is a column vector, so let $x = \bfa\t\bfa > 0$. Let $\sigmaC = x^{1/2} (e^\epsII-1)^{-1/2}$ and let $\sigmaE = x^{1/2} (e-1)^{-1/2}$, and notice $\sigmaC > \sigmaE > 0$. Let $\cC = \frac\epsII2$ and let $\bud \geq \cE = 1$. Since there is only one candidate location, the two sensor configurations to consider are $[\{1\}, \emptyset]$ and $[\emptyset,\{1\}]$, which correspond to placing a cheap sensor or an expensive sensor, respectively, at the candidate location $1$. By \Cref{prop:I+AAT} and \Cref{lemma:matrix-det-lemma},
\begin{align*}
&\Phi_D([\{1\}, \emptyset]) = \log\det(\bI + \sigmaC^{-2} \bfa\bfa\t) = \log(1 + \sigmaC^{-2}x) = \log(1+x^{-1}(e^\epsII-1)x) = \epsII,
\\&
\Phi_D([\emptyset,\{1\}]) = \log\det(\bI + \sigmaE^{-2} \bfa\bfa\t) = \log(1 + \sigmaE^{-2}x) = \log(1+x^{-1}(e-1)x) = 1.
\end{align*}
As $\epsII < 1$, the optimal solution is $S_\mathrm{opt}=[\emptyset,\{1\}]$. However, because the marginal gains are
\begin{align*}
&\cC^{-1}\left(\Phi_D([\{1\}, \emptyset])-\Phi_D([\emptyset, \emptyset])\right) = \frac2\epsII (\epsII-0) = 2,
\\&\cE^{-1}\left(\Phi_D([\emptyset,\{1\}])-\Phi_D([\emptyset, \emptyset])\right) = 1(1-0) = 1,
\end{align*}
the greedy solution is $S_\mathrm{greedy}=[\{1\},\emptyset]$. Then, since $\Phi_D(S_\mathrm{greedy}) = \epsII$ and $\Phi_D(S_\mathrm{opt}) = 1$, we have the result $\Phi_D(S_\mathrm{greedy}) \leq \epsII\Phi_D(S_\mathrm{opt})$.
\end{proof}
It is important to emphasize that the conclusion of \Cref{prop:counterexample} is independent of the training data.
That is, regardless of the properties of the data or the resulting behavior of the posterior covariance as a function of the sensor configuration, one can construct problem instances for which the greedy algorithm performs arbitrarily poorly.
The proof highlights a fundamental limitation of the greedy strategy: it may favor the placement of a cheap sensor over an expensive alternative, even in situations where the latter yields a substantially larger improvement in the objective.

Such behavior is clearly suboptimal in settings where sufficient budget remains to place a higher-fidelity sensor and where no additional sensor placements are possible thereafter, as occurs, for example, when only a single candidate location is available.
Nevertheless, despite this unfavorable worst-case behavior, the greedy algorithm exhibits strong empirical performance in the numerical experiments reported in \Cref{sec:experiments}.

\section{Iterative Selection} 
\label{sec:iterative}
We now introduce a novel iterative framework for multifidelity sensor placement and analyze its computational cost.  

\subsection{Algorithm} 
The approach consists of two main components. First, we identify the set of admissible allocations $(\kc,\ke)$ of cheap and expensive sensors that satisfy the budget constraint. Second, for each admissible allocation, we determine the corresponding sensor locations by alternately optimizing over one fidelity class at a time, selecting all sensors of a given fidelity before switching to the other. 
The algorithm implements the two components as follows.

\myparagraph{Phase 1: Pruning feasible allocations} \label{iter_phase_I} 
We construct a candidate set $\mathcal K$ of candidate sensor allocations $(\kc,\ke)$, pruned from the set of feasible allocations based on \Cref{add_sensor,lower_variance_sensor}.
When the total budget is large relative to sensor costs, enumerating all feasible allocations is computationally intractable. To reduce the search space, we employ some of the results presented in \Cref{subsec:properties}, namely \Cref{add_sensor,lower_variance_sensor}.

Informally, \Cref{add_sensor} states that adding a sensor can never reduce the objective function $\Phi_D$. Hence, any sensor allocation that leaves enough unused budget to add another sensor can be discarded, as it is feasible to add another sensor, and such an addition cannot reduce the objective function. In particular, we ignore all sensor allocations $(\kc,\ke)$ satisfying $\cC \kc + \cE \ke \leq \bud - \cC$, as such allocations leave enough budget to add a cheap sensor. Similarly, \Cref{lower_variance_sensor} states that, for any given location, an expensive sensor there is at least as informative as a cheap sensor there. Thus, we can discard all sensor allocations that call for at least one cheap sensor and leave enough remaining budget to replace a cheap sensor with an expensive one. That is, we ignore all $(\kc,\ke)$ satisfying both $\kc > 0$ and $\cC \kc + \cE \ke \leq \bud - (\cE-\cC)$.

More concretely, we initialize $\K = \emptyset$, and then, for each fixed feasible value of $\ke$, we consider only the maximal value of $\kc$ subject to this $\ke$, as any lesser value of $\kc$ can be discarded by \Cref{add_sensor}. We then apply \Cref{lower_variance_sensor} to determine if the pair $(\kc,\ke)$ can be discarded. If it cannot be discarded, that is, if $\kc = 0$ or $\cC \kc + \cE \ke > \bud - (\cE - \cC)$, we add $(\kc,\ke)$ to $\mathcal K$. 
Since we only consider one value of $\kc$ for each feasible value of $\ke$, after this procedure, $\K$ contains at most one candidate allocation for each feasible $\ke$. Since we must have $0 \leq \ke \leq {\bud}/{\cE}$, at the end of the process we have the bound $|\mathcal K| \leq \floor{1+{\bud}/{\cE}}$. In contrast, the total number of feasible allocations is $\sum_{\ke = 0}^{\floor{\bud/\cE}} \floor{1 + ({\bud  - \cE \ke})/{\cC}}$, since, given a fixed feasible value of $\ke$, there are $\floor{1 + ({\bud  - \cE \ke})/{\cC}}$ feasible values of $(\kc, \ke)$. Thus, this method of candidate generation drastically cuts down the search space of allocations.
This is illustrated in \Cref{tab:pruning} with different values of the costs $\cC$ and $\cE$ and the budget fixed at $\bud=100$. As can be seen, the number of candidate allocations is much smaller than the total number of possible candidates. Furthermore, the upper bound tracks the number of candidate allocations well. 
\begin{table}[H]
    \renewcommand{\arraystretch}{1.3}
    \setlength{\tabcolsep}{8pt}
    \begin{tabular}{c c c c}
        \toprule$\bm{\cC/ \cE}$ & \textbf{\# Feasible allocations } & $\bm{|\mathcal K |}$
        & $\mathbf{\floor{1 + \bud/\cE}}$\\
        \midrule
        1/2 & 2601 & 51 & 51 \\

        2/3 & 884 & 18 & 34 \\

        3/5 & 364 & 14 & 21 \\

        5/11 & 107 & 10 & 10 \\
        \bottomrule
    \end{tabular}
        \centering
    \caption{Comparison of feasible allocations, candidate allocations, and the upper bound on candidate allocations. The budget is fixed at $\bud =100$.}
    \label{tab:pruning}
\end{table}
\begin{algorithm}[!ht]
\caption{Subroutine: Greedy Selection of Sensors}
\label{subroutine:greedyselect}

\begin{algorithmic}[1]
\item[\textbf{Input:}] Fidelity $j \in \{\ch,\exp\}$ of sensors to select, number $k$ of sensors to select, selection set $S_{\rm prev}$ of existing sensors of the other fidelity. 

\item[\textbf{Output:}] Selection set $S$ of $k$ sensors of fidelity $j$. 

\STATE \textbf{function} \textsc{GreedySelect}($j, k, S_{\rm prev}$)
\STATE Initialize $S \gets \emptyset$
\WHILE{$|S| < k$}
    \STATE $I \gets \{1,\dots,M\} \setminus (S \cup S_{\rm prev})$ \hfill \codecomment{Available locations}
    \IF{j = $\ch$}
        \STATE $\displaystyle i^* \gets \argmax_{i \in I} \Phi_D([S \cup \{i\}, S_{\rm prev}]) - \Phi_D([S, S_{\rm prev}])$ 
    \ELSE
        \STATE $\displaystyle i^* \gets \argmax_{i \in I} \Phi_D([S_{\rm prev}, S \cup \{i\}]) - \Phi_D([ S_{\rm prev}, S])$
    \ENDIF
    \STATE $S \gets S \cup \{i^*\}$
\ENDWHILE
\RETURN $S$
\end{algorithmic}
\end{algorithm}
\myparagraph{Phase 2: Sensor placement} 
\label{iter_phase_II}
After constructing the pruned candidate set $\mathcal K$, we iterate over each allocation $(\kc,\ke) \in \mathcal K$. For each allocation $(\kc,\ke)$, we alternately optimize over the cheap and the expensive sensors using the greedy algorithm. We need a slightly modified version of the greedy approach which is detailed in \Cref{subroutine:greedyselect}. Our implementation accelerates this algorithm using the Sherman--Morrison updates, but for brevity, we do not include the details here.

We initialize $\Sc^{(0)} = \emptyset$ and obtain $\Se ^{(0)}$ by applying the single-fidelity greedy algorithm described in \Cref{subsec:single-fidelity} with parameters $\sigma = \sigmaE$ and $k = \ke$. Then, at iteration $t \ge 1$, we consider the objective function
\[ \Psi_{t,\rm ch}(\Sc) = \Phi_D\left(\bI + (\bAe\bSe^{(t-1)})(\bAe\bSe^{(t-1)})\t+ (\bAc\bSc) (\bAc\bSc)\t\right),\]
and compute $\Sc^{(t)}$ by greedily optimizing $\Psi_{t,\rm ch}$ in $\Sc$, subject to the constraints that $|\Sc| = \kc$ and $\Sc \cap\Se^{(t-1)} = \emptyset$. This can be done using \Cref{subroutine:greedyselect} with inputs $j = \ch$, $k=\kc$, and $S_{\rm prev} = \Se^{(t-1)}$. Then, consider the objective function
\[ \Psi_{t,\rm exp}(\Se) = \Phi_D\left(\bI + (\bAe\bSe)(\bAe\bSe)\t+ (\bAc\bSc^{(t)}) (\bAc\bSc^{(t)})\t\right),\]
and compute $\Se ^{(t)}$ by greedily optimizing $\Psi_{t,\rm exp}$ in $\Se$ subject to the constraints that $|\Se| = \ke$ and $\Se  \cap \Sc^{(t)} = \emptyset$. As before, this can be done by calling  \Cref{subroutine:greedyselect} with inputs $j = \exp$, $k=\ke$, and $S_{\rm prev} = \Sc^{(t)}$.

After each optimization step (either cheap or expensive), the iterations are terminated if a maximum number of iterations has been reached or if the objective function decreases. In our experiments, the algorithm typically terminates well before reaching the prescribed maximum number of iterations, with continued iterations yielding no substantial subsequent increase in the objective function. If the criterion is not satisfied, we continue by moving to the next greedy optimization. 
If the criterion is satisfied, we stop and set
$[\Sc^{(\kc,\ke)},\Se^{(\kc,\ke)}]$
to be a configuration achieving the greatest objective value among all configurations for the current allocation $(\kc,\ke)$ whose objective value has been computed.

Finally, after carrying out this process for all candidates $(\kc,\ke)$, we have one sensor configuration $\left[\Sc^{(\kc,\ke)},\Se^{(\kc,\ke)}\right]$ for each candidate allocation $(\kc,\ke) \in \mathcal K$. 
We identify the candidate configuration that maximizes the objective function and select it as the near-optimal solution. 

The complete pseudocode for both phases in the procedure is detailed in \Cref{alg:itergreedy}, provided in \Cref{sec:appendix_algo}.

Before proceeding, we note that, in contrast to the greedy algorithm, extending this approach to an arbitrary number of sensor fidelities is not entirely trivial, as the generation of candidate allocations in Phase 1 relies heavily on the two-class assumption. Thus, we give a detailed summary of such an extension, along with an analysis of its computational cost, in \Cref{sec:appendix_extension}.

\subsection{Computational cost} 
Phase~1 of \Cref{alg:itergreedy} constructs a set $\mathcal K$ of candidate sensor allocations using
\Cref{add_sensor,lower_variance_sensor}. The resulting set contains at most
$1 + \bud/\cE$ candidate pairs $(\kc,\ke)$. Since each candidate is generated with constant work,
the total cost of Phase~1 is $\bigO{\bud/\cE}$ flops.

In Phase~2, the algorithm processes each candidate $(\kc,\ke)\in\mathcal K$ independently.
For a given candidate, we first greedily initialize $\ke$ expensive sensors, which requires
$\bigO{\ke M\ell}$ flops.
The algorithm then performs at most $T$ iterations of greedy refinement.
Each refinement iteration consists of two steps: 
(i) selecting $\kc$ cheap sensors conditional on the
current set of expensive sensors, and 
(ii) selecting $\ke$ expensive sensors conditional on the
current set of cheap sensors.

Using the Sherman--Morrison-based implementation described in \Cref{subsec:sherman},
each greedy selection requires $\bigO{(\kc+\ke)M\ell}$ flops.
Indeed, forming the matrices $\bD_\ch$ and $\bD_\exp$ associated with the
currently selected sensors costs $\bigO{\ke M\ell}$ flops, while the subsequent greedy
selection of $\kc$ cheap sensors contributes an additional $\bigO{\kc M\ell}$ flops, see \Cref{para:greedy-sm-cost}. 

Consequently, the total cost for processing a single candidate $(\kc,\ke)$ in Phase~2 is
$\bigO{(\kc+\ke)T M\ell}$ flops. Since all admissible allocations satisfy\footnote{It is possible to calculate the time complexity by summing over the allocations in $\mathcal K$ instead of applying this bound. One can verify that doing so does not give a tighter Big-O time complexity, although we omit that calculation due to its length.}
$\kc+\ke \le \bud/\cC$, it follows that the cost per candidate is bounded by 
$\bigO{(\bud/\cC)T M\ell}$ flops. 
Since $|\mathcal K| = \bigO{\bud/\cE}$ and each candidate requires
$\bigO{(\bud/\cC) T M\ell}$ flops, Phase~2 has total cost
$\bigO{\bud^2/(\cC\cE)\, T M\ell}$. This term dominates the cost of Phase~1,
and therefore the overall computational complexity of
\Cref{alg:itergreedy} is
\[
\bigO{\frac{\bud^2}{\cC\cE}\, T M\ell} \> \text{flops}.
\]
\subsection{Comparison with greedy}
The aforementioned bound $\kc+\ke \le \bud/\cC$ implies that the computational complexity of \Cref{mfgreedySM} is $\bigO{\frac{\bud}{\cC} M \ell}$ flops, meaning that the iterative algorithm is asymptotically at least as expensive as the greedy algorithm. Thus, we now justify the use of the iterative algorithm over the greedy approach despite the increased computational expense. As demonstrated in \Cref{subsec:performance}, the greedy algorithm may favor the placement of a cheap sensor over an expensive one, even when the latter would lead to a substantially larger improvement in the objective, which can result in a clearly suboptimal configuration, particularly when the number of candidate sensor locations is very small. More generally, a principal drawback of the greedy algorithm is that it considers only one allocation $(\kc,\ke)$ (namely, the one called for by the sensor configuration it builds up), which may be an extremely poor allocation, i.e., even the optimal sensor configuration among those calling for this allocation may be a highly suboptimal solution to the cost-constrained problem we consider. 

In contrast, the iterative algorithm considers all feasible allocations $(\kc,\ke)$ that may be optimal, ensuring that an optimal allocation is at least considered in the process of choosing a sensor configuration, thus addressing the described issue faced by the greedy algorithm. This advantage of the iterative algorithm over the greedy approach is illustrated by considering the example problem instance in \Cref{prop:counterexample}, in which Phase 1 of the iterative algorithm would discard the allocation $(\kc,\ke)=(1,0)$ because $\cC\kc + \cE\ke = \cC \leq (\bud - \cE)+\cC=\bud - (\cE-\cC)$. Therefore, at the end of Phase 1, the set of candidate allocations would be $\mathcal K = \{(0,1)\}$, and, since the only configuration with this allocation is the optimal configuration $[\emptyset, \{1\}]$, the iterative algorithm would choose the optimal solution. 
\section{Numerical Experiments}
\label{sec:experiments}
In this section, we provide a broad set of experiments to test the performance of the algorithms presented in \Cref{subsec:sherman} and \Cref{sec:iterative}. Throughout this section, we report the number of iterations performed by the Iterative algorithm (\Cref{alg:itergreedy}). Here, one iteration consists of a greedy optimization step over the cheap sensors followed by a greedy optimization step over the expensive sensors. Recall from \Cref{sec:iterative} that the parameter $T$ denotes the prescribed maximum number of iterations allowed per candidate allocation. In all experiments, we set $T=1000$ although the algorithm typically terminates much earlier.
We let $T_\mathrm{obs}$ denote the total number of actual iterations carried out across all candidate allocations in $\K$.
\myparagraph{Datasets} 
We evaluate the proposed methods on two benchmark datasets. The first is the National Oceanic and Atmospheric Administration (NOAA) Sea Surface Temperature (SST) dataset, which consists of 1713 weekly temperature fields (in degrees Celsius) collected between 1990 and 2022 \cite{noaa_oisst_v2}. Each snapshot is defined on a $360 \times 180$ spatial grid.

The second dataset corresponds to a two-dimensional viscous flow past a cylinder, capturing 1501 snapshots of the characteristic von Kármán vortex street \cite{Guenther17}.
Each snapshot consists of two velocity components, $u$ and $v$, each discretized on a $640 \times 80$ spatial grid. For these experiments, we use only the $u$-component, corresponding to the velocity in the horizontal direction.

For both datasets, the snapshots are partitioned chronologically, with the first $70\%$ used for training and the remaining $30\%$ reserved for testing. A reduced basis $\bPhi \in \R^{N \times \ell}$ is constructed from the training data via a thin singular value decomposition, where $\ell$ is chosen as the smallest dimension capturing $99\%$ of the cumulative singular value energy, i.e., the sum of the squares of the $\ell$ largest singular values is at least $99\%$ of the sum of the squares of all singular values. The prior covariance matrix is then formed according to \eqref{eq:priorCov} with parameter value $\lambda = 0.01$. Reconstruction accuracy is assessed on the testing set, and we report the average relative reconstruction error over all test snapshots in the 2-norm, i.e., $\|\bfu - \widehat{\bfu}\|_2 / \|\bfu\|_2$. For visualization purposes, we show a representative test snapshot from the SST dataset recorded on June 14, 2020, together with a mid-trajectory snapshot from the vortex shedding sequence in the flow past a cylinder dataset; see \Cref{tab:allocs}.

\begin{table}[!ht]
  \centering
  \begin{tabular}{c c c c c c c}
    \toprule
    \textbf{Dataset} & \bm{$M/N$} &\textbf{\# modes ($\bm{\ell}$)}  & \textbf{ALG} & $\bm{\kc}$ & $\bm{\ke}$ & \textbf{RE} \\
    \midrule
    \multirow{2}{*}{SST}
      & \multirow{2}{*}{44219/44219} & \multirow{2}{*}{178} & Greedy  & 20 &  5 & 0.1038 \\
      &        &         & Iterative &  1 & 10 & 0.0978 \\
    \midrule
    \multirow{2}{*}{\shortstack{Flow\\past Cylinder}}
      & \multirow{2}{*}{51200/51200} & \multirow{2}{*}{35} & Greedy   &  28 &  3 & 0.0864 \\
      &        &     & Iterative &  9 & 8 & 0.0864 \\
    \bottomrule
  \end{tabular}
  \caption{Sensor allocations and average relative errors (over the testing set) for each dataset.}
  \label{tab:allocs}
\end{table}

\begin{figure}[!ht]
    \centering
    \includegraphics[width=\linewidth] {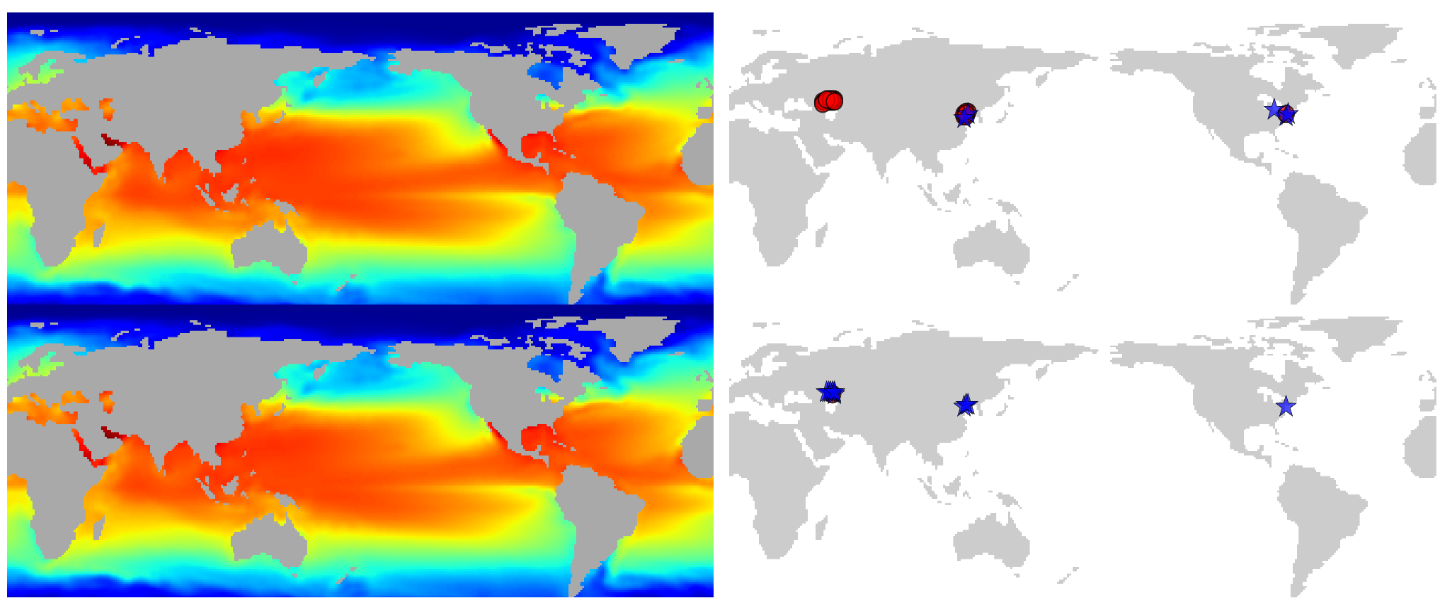}
    \caption{MAP reconstructions of sea surface temperature using sensors selected by the greedy (top) and iterative (bottom) algorithms. Both methods employ the same settings, yielding comparable relative errors of about $11\%$. For the iterative method, the candidate set contains $|\mathcal K| = 11$ sensor allocations, and the algorithm performs a total of $T_\mathrm{obs}=2$ iterations. The left panels show the MAP estimate, while the right panels display the corresponding sensor locations. Cheap sensors are denoted by circles and expensive sensors by stars.} 
    \label{fig:sst-recon-zoomed-vertical}
\end{figure}

\subsection{Reconstructions}
\label{subsec:reconstructions}
In this section, we present reconstructions using sensor configurations selected by the proposed algorithms.
The aim of this is twofold: 
First, they provide a concrete illustration of the state estimation problem described in \Cref{subsec:state}.
Second, although the D-optimality criterion is designed to maximize expected information gain rather than to directly minimize the reconstruction error of the MAP estimate, it is not \emph{a priori} clear that sensor configurations with higher D-optimality necessarily yield more accurate reconstructions. The experiments that follow provide empirical evidence that the proposed algorithms identify sensor locations leading to accurate recovery of the underlying system state.

In these experiments, the total budget is set to $\bud = 1000$,
with sensor costs $\cC = 25$ and $\cE = 96$, and measurement noise standard deviations $\sigmaC = 0.02$ and $\sigmaE = 0.01$ for the cheap and expensive sensors, respectively.  
As shown in \Cref{fig:sst-recon-zoomed-vertical}, for the SST dataset, both the greedy and iterative selection strategies predominantly place sensors in landlocked regions and lakes. Since the columns of $\bPhi$ span a low-dimensional subspace that captures the dominant variability in the training data, the algorithms prioritize locations exhibiting strong local variability, which often occur over inland water bodies.  
\Cref{fig:flow-recon-zoomed} shows that for the flow past a cylinder dataset, the selected sensors cluster near the cylinder, where the flow exhibits the largest spatial gradients. The small relative errors observed in each case confirm that the resulting configurations yield reconstructions that are accurate and nearly optimal within the prescribed budget, providing evidence that, although the presented algorithms may choose significantly different sensor configurations, such as in this experiment, they are effective in choosing sensor configurations that yield accurate reconstructions.
\subsection{Experiments} 
\label{subsec:exp}
We now turn to quantitative experiments on the behaviors of our algorithms as the inputs vary. In \Cref{subsubsec:exp1} we examine how the greedy algorithm's allocation and placement of sensors change as the sensor costs and measurement noise standard deviations vary. Then, in \Cref{subsubsec:exp2}, we examine how our algorithms' allocation and placement of sensors change as the budget increases. Finally, in \Cref{subsubsec: exp3} we compare the D-optimality objective values of sensor configurations chosen by our algorithms against random designs.

\begin{figure}[!ht]
    \centering
    \includegraphics[width=\linewidth]{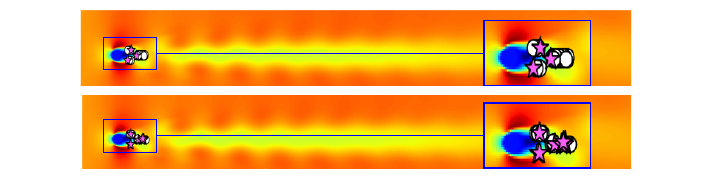}
 \caption{MAP reconstructions of flow past a cylinder obtained using greedy (top) and iterative (bottom) sensor selection. Both approaches achieve comparable relative errors of approximately $9\%$. For the iterative method, the candidate set consists of $|\mathcal K|=11$ sensor allocations, and the algorithm performs a total of $T_\mathrm{obs}=9$ iterations. The zoomed insets highlight sensor placements concentrated near regions of high flow variability.}
      \label{fig:flow-recon-zoomed}
\end{figure}

\subsubsection{Experiment 1: Costs and Measurement Noise Standard Deviations}
\label{subsubsec:exp1}

In this section, we analyze how variations in sensor costs and measurement noise standard deviations influence the greedy algorithm's allocation and placement strategy, while keeping the total budget fixed. 
The goal of this experiment is to determine under what conditions the algorithm favors one sensor fidelity over the other, and to identify the critical regime where it is indifferent between them.

We note that, in our implementation, such indifference results in a preference towards cheap sensors, as our implementation defaults to the lower-cost option when both provide equal marginal information gain.

We consider three regimes determined by the comparison of the cost ratio ${\cC}/{\cE}$ and the noise variance ratio ${\sigmaE^2}/{\sigmaC^2}$: ${\cC}/{\cE} > {\sigmaE^2}/{\sigmaC^2}$, ${\cC}/{\cE} < {\sigmaE^2}/{\sigmaC^2}$, and the critical case ${\cC}/{\cE} = {\sigmaE^2}/{\sigmaC^2}$. The sensor configurations selected by the greedy algorithm for these regimes are shown in \Cref{fig:cost-noise-ratio-gridlwbf}, and the corresponding parameter values and sensor counts are reported in \Cref{tab:cost-noise-ratio}.
When the cost ratio ${\cC}/{\cE}$ is larger than the critical value ${\sigmaE^2}/{\sigmaC^2}$, the algorithm favors expensive sensors due to their higher information gain per unit cost. Conversely, when the cost ratio is smaller, cheap sensors dominate the selection. At the critical ratio, the algorithm is indifferent between the two fidelities; however, in our implementation, it defaults to favoring cheap sensors due to the cost-based tie-breaking rule. A heuristic explanation of these results is given in \Cref{sec:heuristic}.
\begin{figure}[!ht]
    \begin{subfigure}[t]{0.48\textwidth}
        \centering
        \adjincludegraphics[width=\textwidth, height=2cm, trim={0 {.6\height} 0 0}, clip]{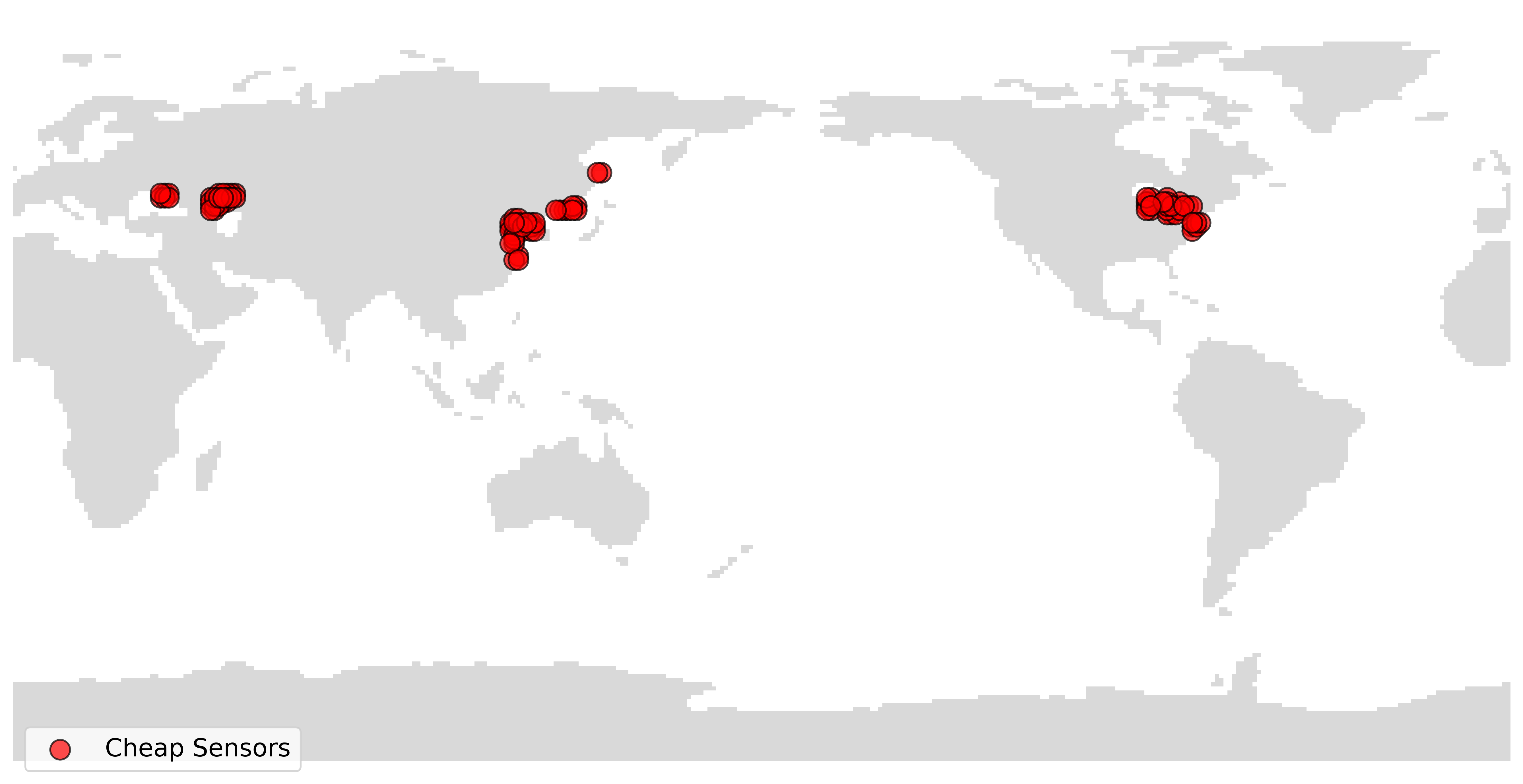}
        \caption{\textbf{Case:} $\displaystyle \frac{\cC}{\cE} = \frac{\sigmaE^2}{\sigmaC^2}$}
        \label{fig:cost_ratio_equal}
    \end{subfigure}
    \hfill
    \begin{subfigure}[t]{0.48\textwidth}
        \centering
        \adjincludegraphics[width=\textwidth, height=2cm, trim={0 {.6\height} 0 0}, clip]
        {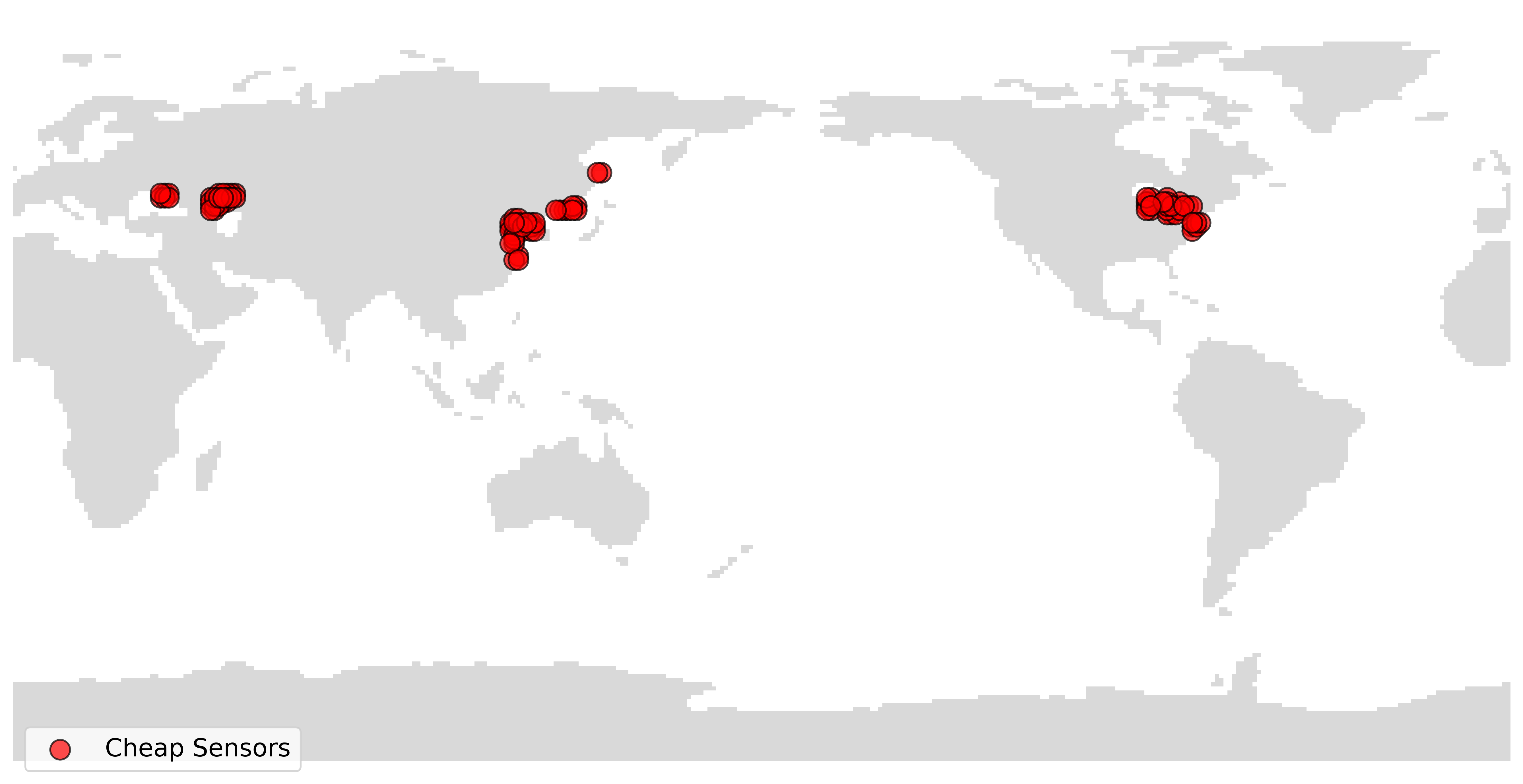}
        \caption{\textbf{Case:} $\displaystyle \frac{\cC}{\cE} < \frac{\sigmaE^2}{\sigmaC^2}$}
        \label{fig:cost_ratio_greater}
    \end{subfigure}

    \begin{subfigure}[t]{0.48\textwidth}
        \centering
        
        \adjincludegraphics[width=\textwidth, height=2cm, trim={0 {.6\height} 0 0}, clip]{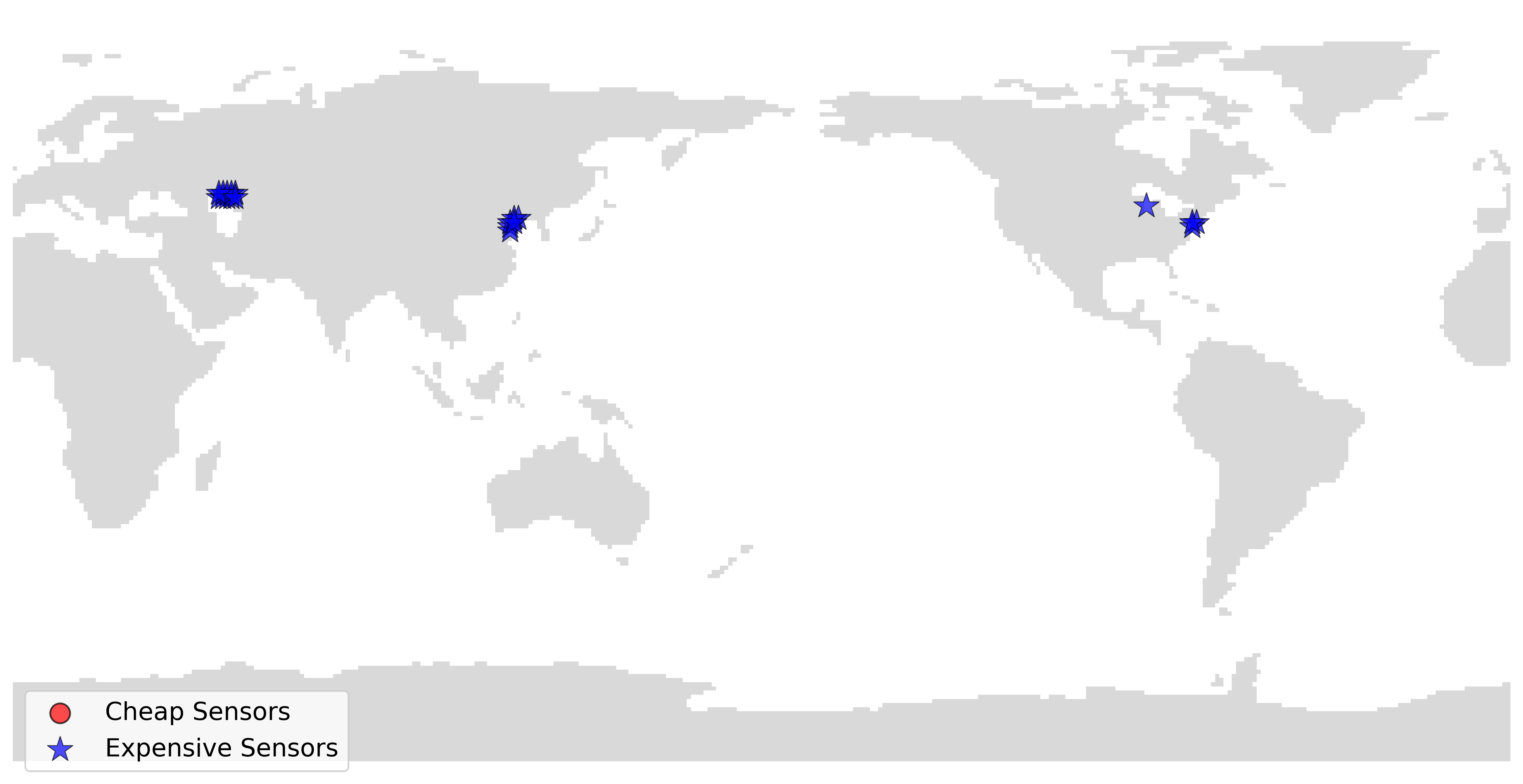}
        \caption{\textbf{Case:} $\displaystyle \frac{\cC}{\cE} > \frac{\sigmaE^2}{\sigmaC^2}$}
        \label{fig:cost_ratio_less}
    \end{subfigure}
    \hfill
    \begin{subfigure}[t]{0.48\textwidth}
        \centering
        \adjincludegraphics[width=\textwidth, height=2cm, trim={0 {.6\height} 0 0}, clip]{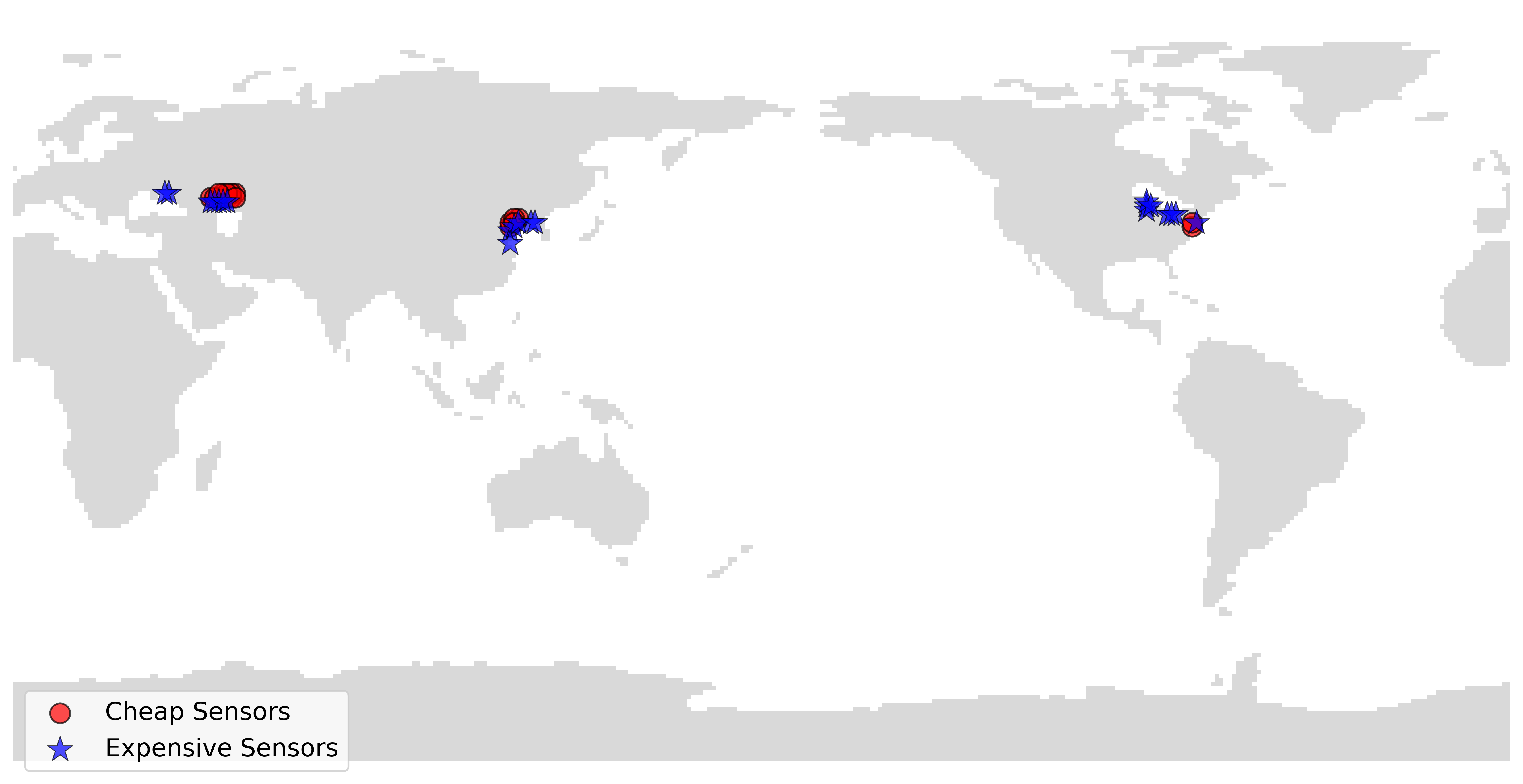}
        \caption{\textbf{Case:} $\displaystyle \frac{\cC}{\cE}  \lessapprox \frac{\sigmaE^2}{\sigmaC^2}$}
        \label{fig:cost_ratio_slightly_less}
    \end{subfigure}

    \caption{Greedy sensor selection results for different costs and measurement noise standard deviations. 
    }
    \label{fig:cost-noise-ratio-gridlwbf}
\end{figure}

\begin{table}[!ht]
    \renewcommand{\arraystretch}{1.3}
    \setlength{\tabcolsep}{8pt}
    \begin{tabular}{c c c c l}
        \toprule
         \textbf{Case} 
        & $\bm{\cC/ \cE}$ & $\bm{\sigmaC/ \sigmaE}$ & $\bm{\kc/\ke}$ & \textbf{Description} \\
        \midrule
         \eqref{fig:cost_ratio_equal}
             & 1/4 & 0.02/0.01 & 100/0 & Critical (balanced) ratio \\

        \eqref{fig:cost_ratio_greater}
            & 1/6 & 0.02/0.01 & 100/0 & Cheap sensors dominate \\

       \eqref{fig:cost_ratio_less}
             & 1/5 & 0.05/0.01 & 0/20 & Expensive sensors dominate \\

        \eqref{fig:cost_ratio_slightly_less}
             & 1/3.85 & 0.02/0.01 & 19/21 & Slightly below critical ratio \\
        \bottomrule
    \end{tabular}
        \centering
    \caption{Summary of cases for different costs and measurement noise standard deviations. The budget is fixed at $\bud =100$.}
    \label{tab:cost-noise-ratio}
\end{table}

\subsubsection{Experiment 2: Sensor Allocation and Increasing Budget} 
\label{subsubsec:exp2}
We explore how the allocation of cheap and expensive sensors changes as the total budget $\bud$ increases. For each test, we fix the costs $\cC = 1$ and $\cE = 5$, and measurement noise standard deviations $\sigmaC = 0.02$ and $\sigmaE = 0.01$, corresponding to a regime where ${\cC}/{\cE} < {\sigmaE^2}/{\sigmaC^2}$, so cheap sensors provide higher information gain per unit cost. As shown in \Cref{fig:greedy-iterative-budget-grid}, the  
greedy algorithm selects all cheap sensors, while the iterative algorithm selects a mix of both fidelities, though it tends to favor cheap sensors. The difference in sensor fidelity becomes more pronounced as the budget increases. For comparison, when ${\cC}/{\cE} > {\sigmaE^2}/{\sigmaC^2}$, both algorithms favor expensive sensors and when ${\cC}/{\cE} = {\sigmaE^2}/{\sigmaC^2}$, greedy favors cheap and iterative favors expensive. For the iterative algorithm, the size of the candidate allocation set and the number of refinement iterations remain modest across all budgets, with the algorithm rapidly converging in all cases. Specifically, for $\bud=50$, the candidate set contains $|\mathcal K |=11$ sensor allocations and the algorithm performs a total of $T_\mathrm{obs}=3$ iterations; for $\bud=100$, $|\mathcal K |=21$ and $T_\mathrm{obs}=5$; and for $\bud=200$, $|\mathcal K |=41$ and $T_\mathrm{obs}=16$.

\begin{figure}[!ht]
  \centering
  \setlength{\tabcolsep}{4pt}
  \renewcommand{\arraystretch}{1.0}
  \begin{tabular}{cc}

    \adjincludegraphics[width=0.48\textwidth,
      trim={{.02\width} {.45\height} {.12\width} 0},clip]
      {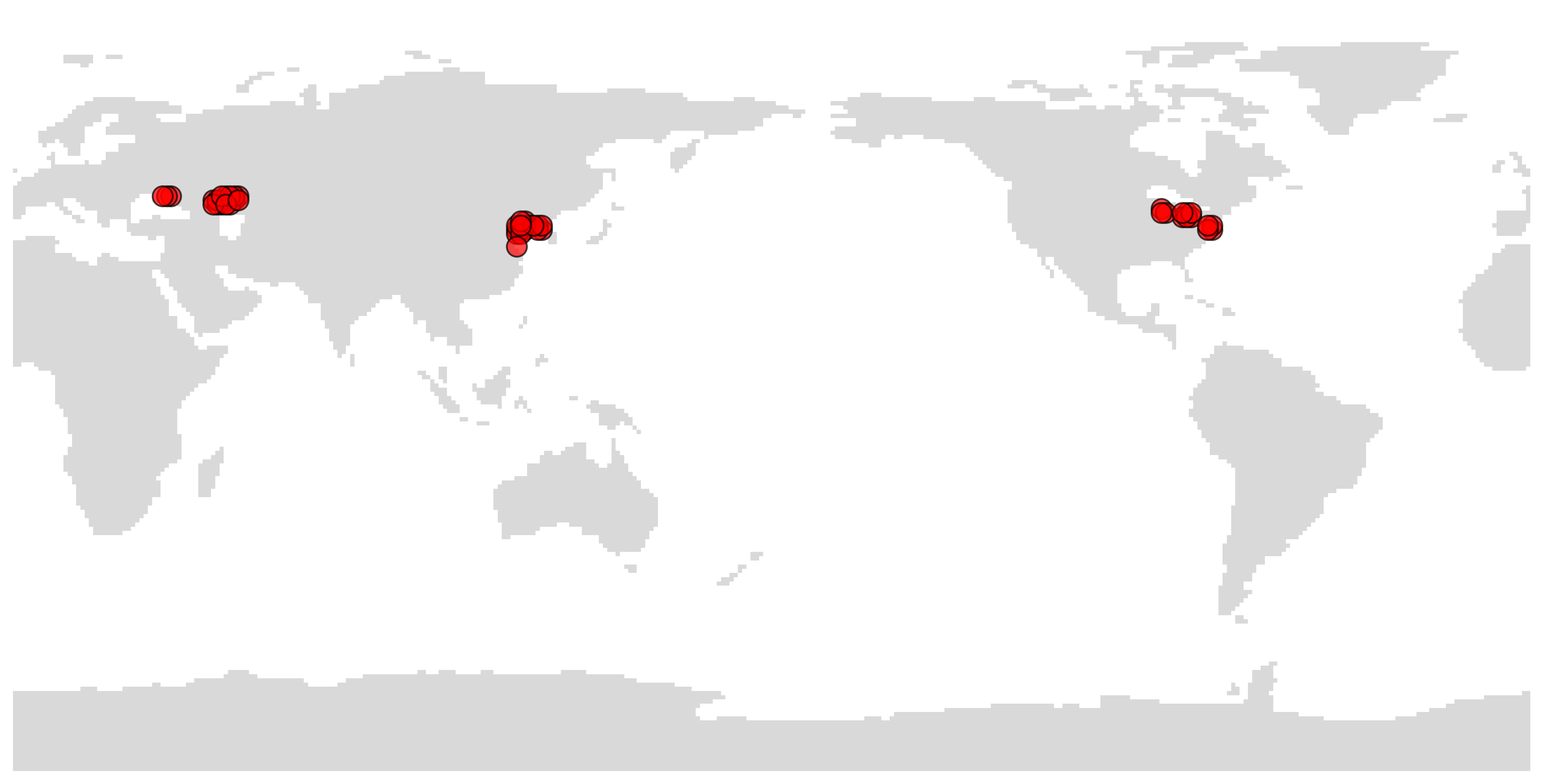} &
    \adjincludegraphics[width=0.48\textwidth,
      trim={{.02\width} {.45\height} {.12\width} 0},clip]
      {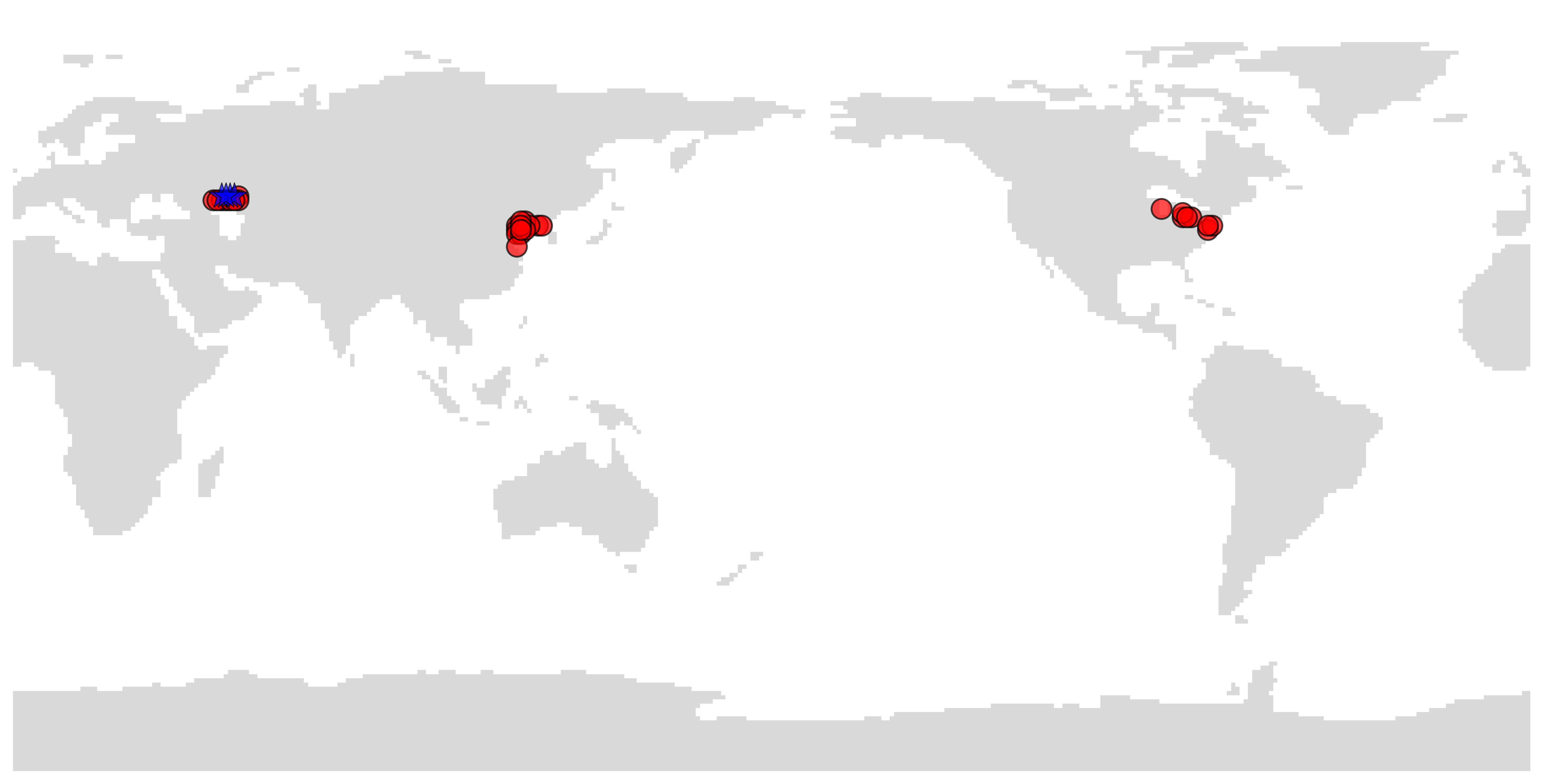} \\
    $ \kc{:}~50,\; \ke{:}~0$  & $ \kc{:}~30,\; \ke{:}~4$ \\[6pt]

    \adjincludegraphics[width=0.48\textwidth,
      trim={{.02\width} {.4\height} {.12\width} 0},clip]
      {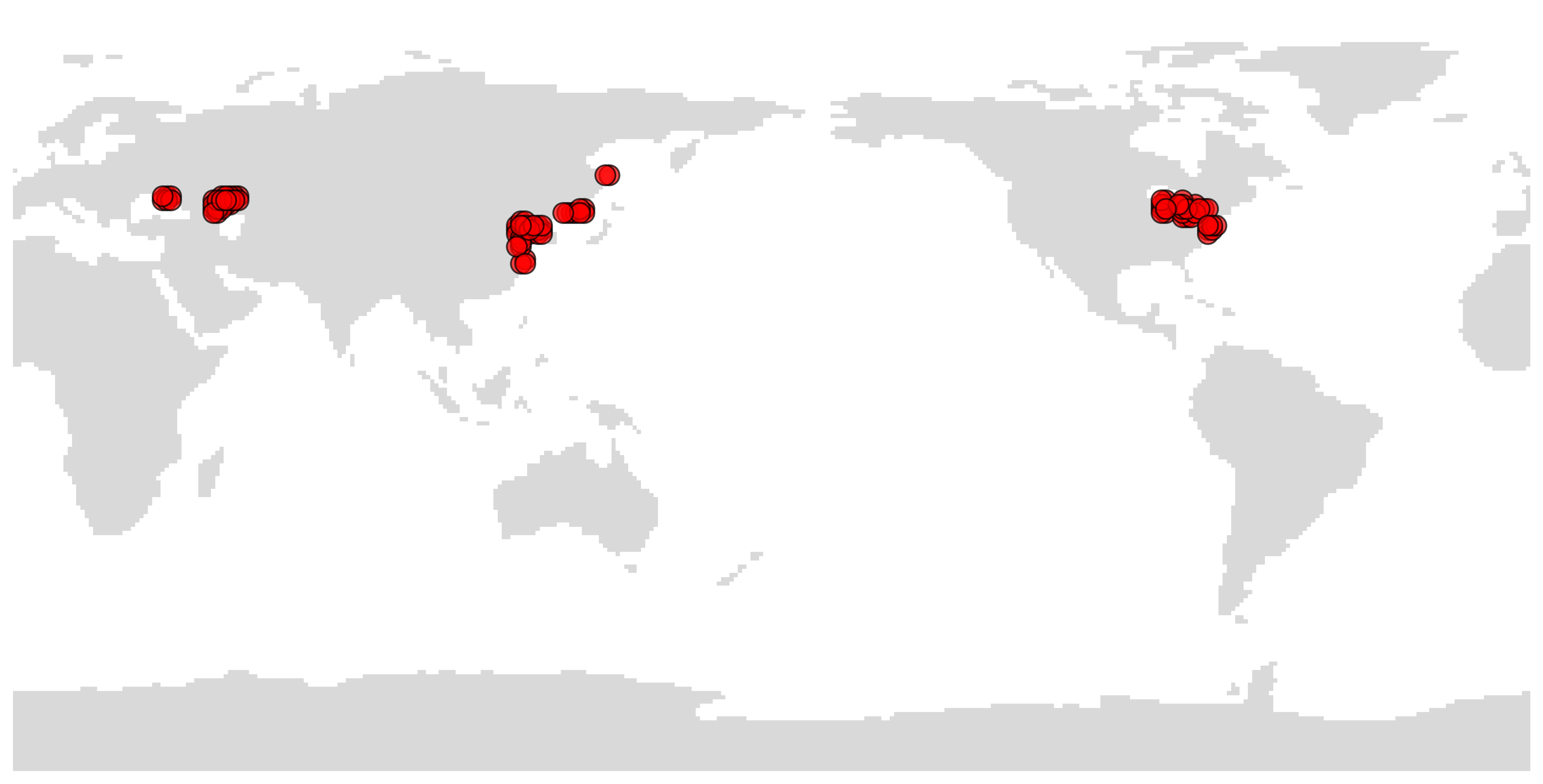} &
    \adjincludegraphics[width=0.48\textwidth,
      trim={{.02\width} {.45\height} {.12\width} 0},clip]
      {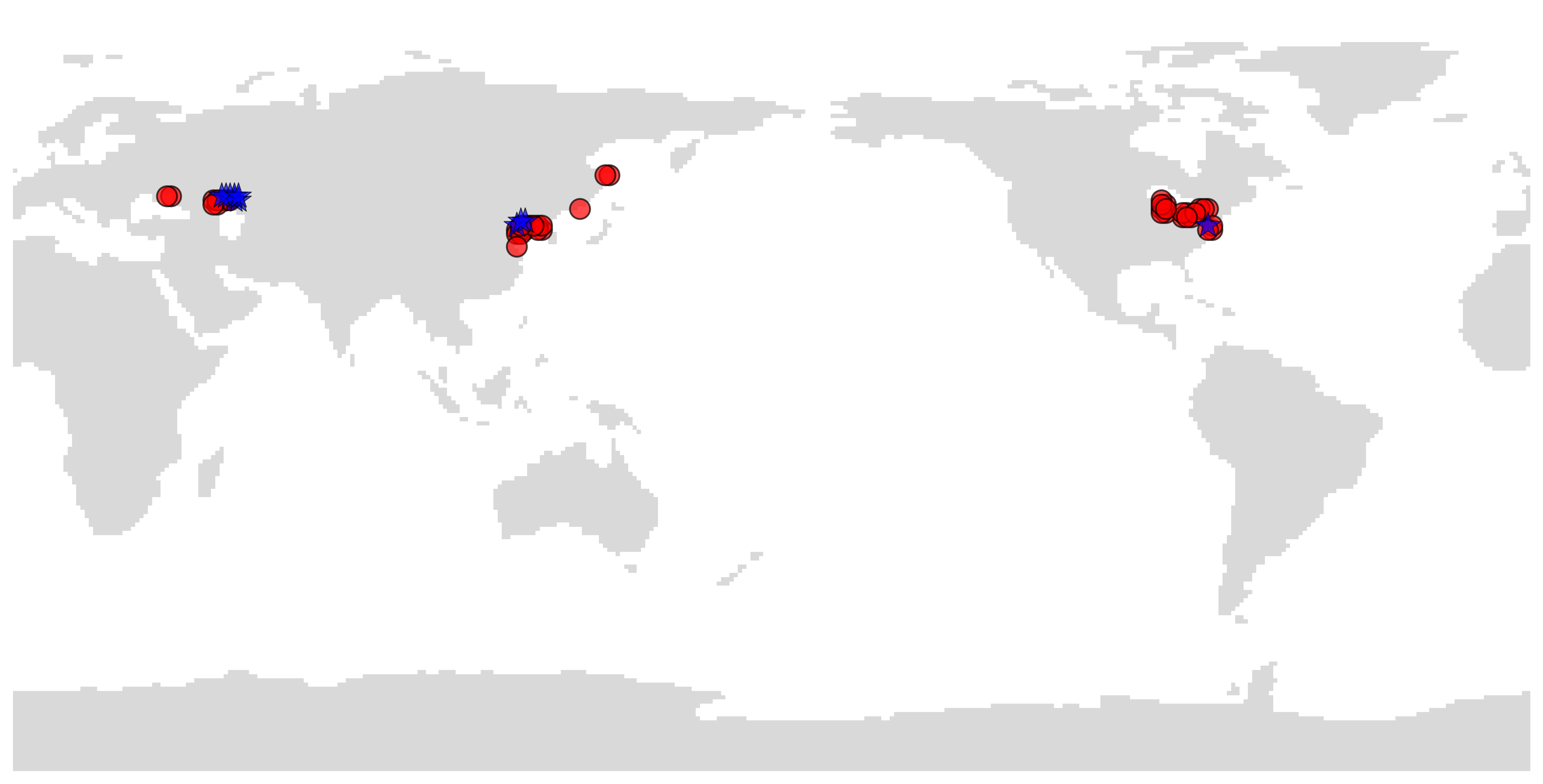} \\
    $ \kc{:}~100,\; \ke{:}~0$& $ \kc{:}~45,\; \ke{:}~11$ \\[6pt]

    \adjincludegraphics[width=0.48\textwidth,
      trim={{.02\width} {.45\height} {.12\width} 0},clip]
      {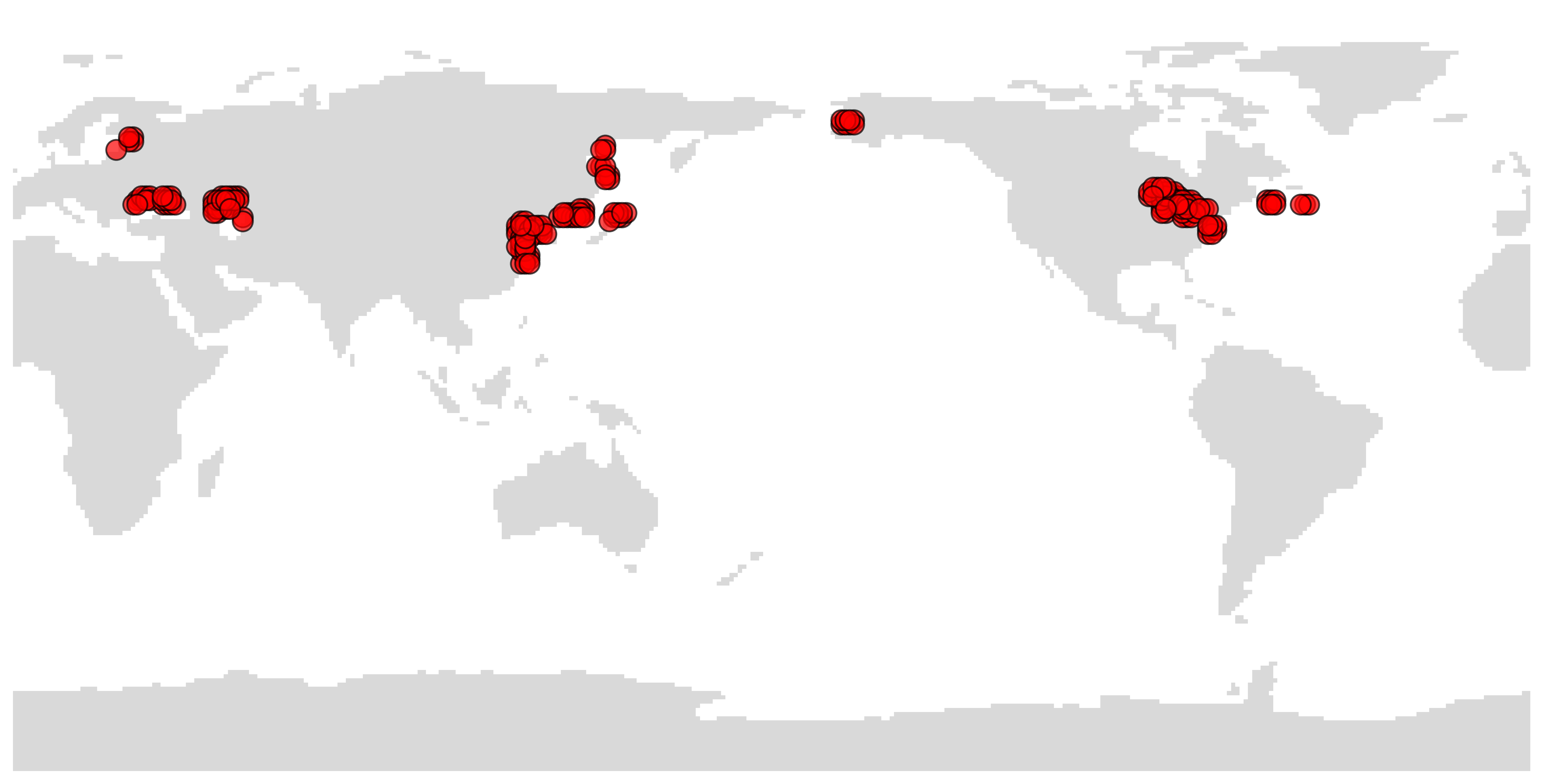} &
    \adjincludegraphics[width=0.48\textwidth,
      trim={{.02\width} {.4\height} {.12\width} 0},clip]
      {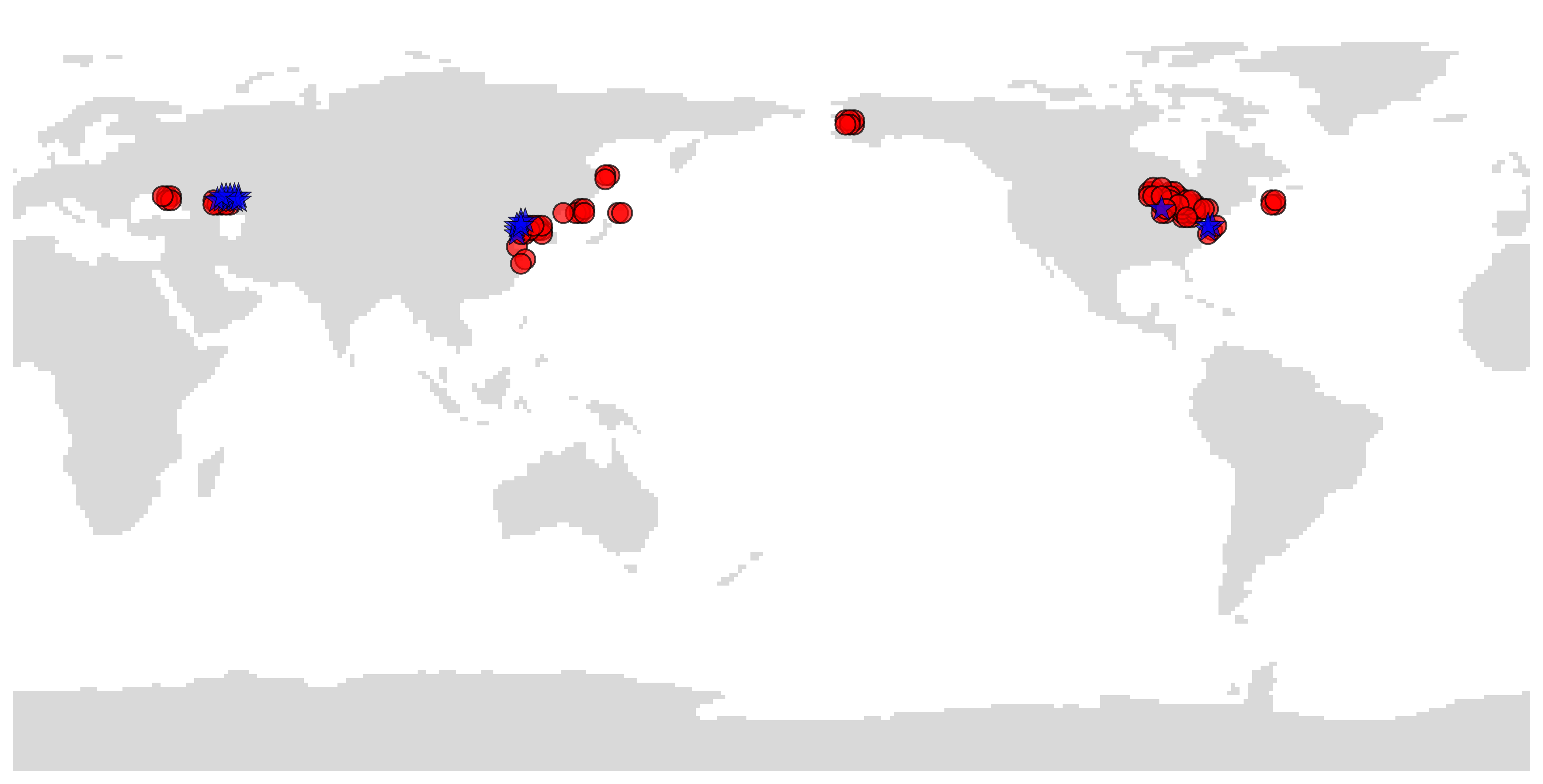} \\
    $ \kc{:}~200,\; \ke{:}~0$ & $ \kc{:}~95,\; \ke{:}~21$

  \end{tabular}

  \caption{Sensor placement results for the greedy (left) and iterative (right) algorithms.
  Rows correspond to increasing  budgets $\bud = 50, 100, 200$ from top to bottom.
  The corresponding $(\kc, \ke)$ values are indicated below each image, with cheap sensors shown in red and expensive sensors shown in blue.}
  \label{fig:greedy-iterative-budget-grid}
\end{figure}

\begin{table}[!ht]
    \vspace{0.5em}
    \centering
    \begin{tabular}{c c c c c c c}
        \toprule
        \textbf{Dataset} 
        & $\bm{\bud}$ 
        & $\bm{\cC/ \cE}$ 
        & $\bm{\sigmaC/ \sigmaE}$ 
        & $\bm{|\mathcal K |}$
        & $\bm{\Phi_D(\text{Greedy})}$ 
        & $\bm{\Phi_D(\text{Iterative})}$\\
        \midrule

        \multirow{2}{*}{SST}
            & 500 & 10/38 & 0.02/0.01 & 14 & 1.5078 & 1.6056  \\
        & 100 & 1/5   & 0.04/0.02 & 21 & 0.8072 & 0.8741\\
        \addlinespace[4pt]\midrule\addlinespace[4pt]

        \multirow{2}{*}{\shortstack{Flow\\past Cylinder}}
            & 500 & 10/38 & 0.02/0.01 & 14 & 12.6884 & 12.9191  \\
        & 100 & 1/5  & 0.04/0.02 & 21 & 8.9895 & 8.9895  \\
        \bottomrule
    \end{tabular}
    \caption{Summary of cases for different D-optimality comparisons.}
    \label{tab:d-opt}
\end{table}

\subsubsection{Experiment 3: Comparison of Algorithms} 
\label{subsubsec: exp3}
For this experiment, we compare the D-optimality objective values of the accelerated greedy algorithm against the iterative selection approach, as well as compare both against random designs as a baseline. To generate random designs, we construct the set $\mathcal K$ of candidate allocations obtained by carrying out Phase~I (\cref{iter_phase_I}) of the iterative selection \Cref{alg:itergreedy}, and for each $(\kc,\ke) \in \mathcal K$, we generate $1000$ random sensor configurations. The D-optimality objective values from these configurations are plotted in histograms. \Cref{tab:d-opt} reports the $\Phi_D$ values achieved by the greedy and iterative algorithms for each dataset under two different budget and noise settings. 

We now report on the number of iterations for the Iterative algorithm (\Cref{alg:itergreedy}). For the SST dataset, $T_\mathrm{obs}=7$ for $\bud=100$ and $T_\mathrm{obs}=1$ for $\bud=500$. For the flow past a cylinder dataset, $T_\mathrm{obs}=21$ and $T_\mathrm{obs}=14$ for $\bud=100$ and $\bud=500$, respectively. These results provide empirical evidence that, under the D-optimality criterion, the proposed iterative selection approach performs better than the greedy method, and that both the iterative and greedy approaches perform considerably better than random designs.
\Cref{fig:alg-comparison-2x2} displays our results for the SST and flow past a cylinder dataset corresponding to $\bud=500$.

\begin{figure}[!ht]
    \centering
    \begin{subfigure}[b]{0.45\textwidth}
        \centering
        \includegraphics[height=0.165\textheight]{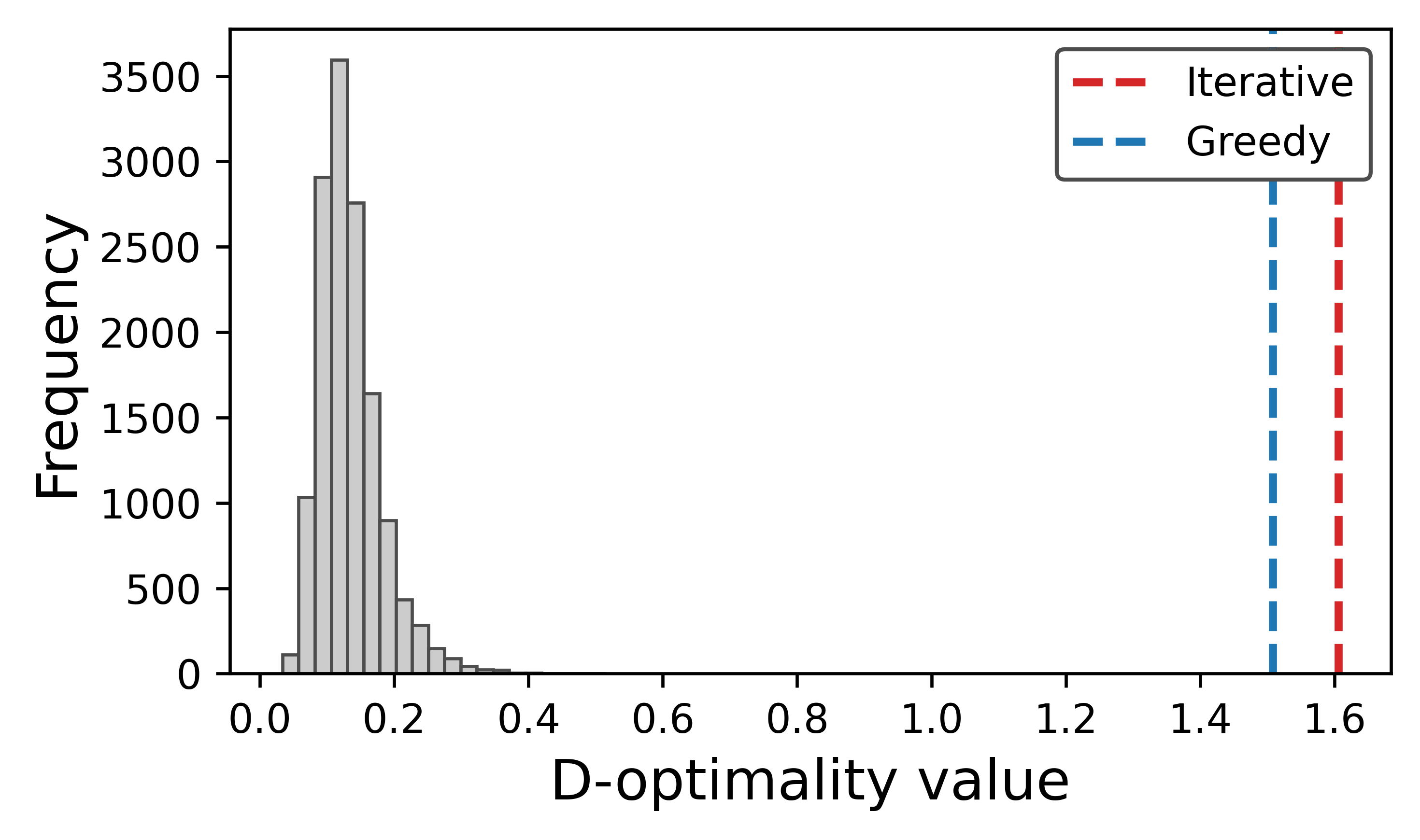}
        \caption{SST dataset}
        \label{fig:figure5a}
    \end{subfigure}
    \hfill
    \begin{subfigure}[b]{0.45\textwidth}
        \centering
        \includegraphics[height=0.165\textheight]{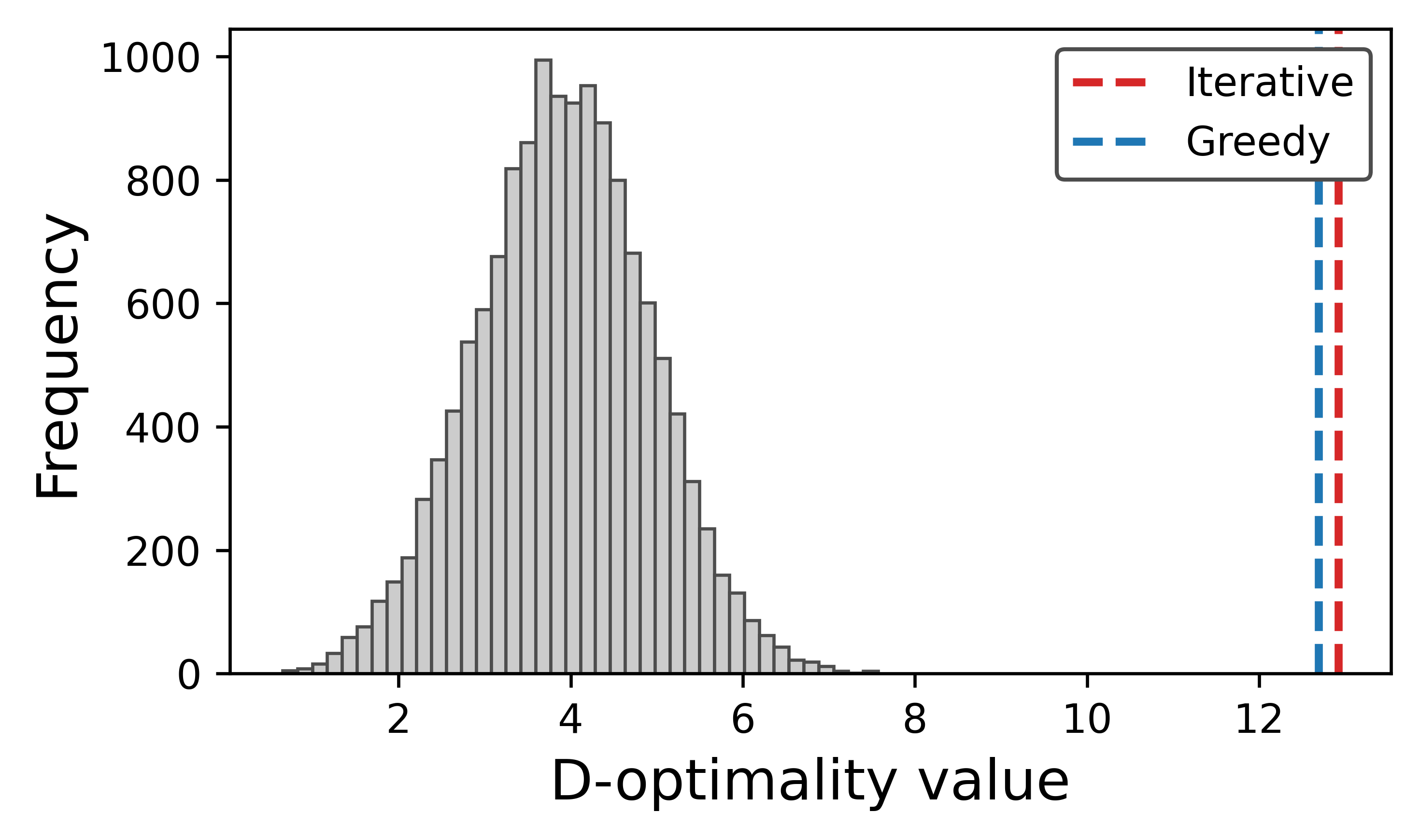}
        \caption{Flow past cylinder dataset}
        \label{fig:figure5b}
    \end{subfigure}
    \caption{D-optimality comparison of random, greedy, and iterative sensor selection at $\bud=500$. Both greedy and iterative selection outperform random selection.}
    \label{fig:alg-comparison-2x2}
\end{figure}

\section{Conclusion}
\label{sec:conclusion}
Optimal data collection has applications in many areas of science and engineering, including medicine, geophysics, and fluid mechanics. We considered this task via sensor placement for the Bayesian inverse problem of estimating the full state of a system from partial and noisy sensor measurements. Extending this sparse sensor selection problem to account for different costs and fidelities, we introduced a multifidelity, cost-constrained formulation of D-optimal design and proposed two computational strategies: a greedy algorithm and an iterative refinement method.

Numerical results show that both algorithms effectively optimize the D-optimality criterion, with the iterative approach consistently performing better than the greedy strategy.
While accurate state reconstruction is not the primary goal of D-optimal design, the greedy and iterative algorithms produce accurate state reconstructions, with the iterative approach performing slightly better.
Across a range of sensor costs and noise levels, the greedy algorithm exhibits intuitive selection behavior by favoring sensors with lower cost-to-noise ratios, whereas the iterative method more aggressively exploits fidelity to enhance information gain.
As the budget increases, the differing selection mechanisms of the two algorithms become increasingly apparent across cost and noise regimes.

Looking ahead, these results motivate the application of the proposed methods to a broader class of state estimation and inverse problems. In addition, the submodularity of the objective in~\eqref{prop:monotone} can be exploited to further reduce the number of evaluations in the greedy search.
Since each candidate’s marginal gain decreases as the selected set grows, one can apply an \emph{accelerated} or \emph{lazy greedy} strategy~\cite{Minoux,Leskovec}, which maintains a priority queue of upper bounds and updates them only when necessary.
This idea, previously explored in D-optimal design for the single-fidelity case in~\cite{Robertazzi}, preserves the theoretical guarantees of the standard greedy algorithm while substantially lowering computational cost.
We aim to explore this in a subsequent work, combining it with the Sherman-Morrison updates to further accelerate large-scale sensor selection. Another natural extension would be to examine our algorithms on an arbitrary number of sensor fidelity levels and costs. An additional interesting direction of future work would be to examine this proposed framework under a different optimality criterion, such as A-optimality, which focus on different uncertainty reductions and could lead to different placement of sensors.

\appendix
\section{Implementation of the Iterative Algorithm}
\label{sec:appendix_algo}

In this section, we provide the pseudocode for the proposed method. 
Algorithm \ref{alg:itergreedy} outlines the iterative greedy refinement 
strategy.

\begin{algorithm}[!ht]
\caption{Multifidelity Iterative D-Optimal Sensor Selection}
\label{alg:itergreedy}
\begin{algorithmic}[1]
\item[\textbf{Input:}]
Sensor costs $c_{\ch} < c_{\exp}$; measurement noise standard deviations $\sigma_{\ch} > \sigma_{\exp}$;
budget $\bud$; prior covariance $\bSigma_{\mathrm{pr}}$;
basis $\bPhi \in \mathbb{R}^{N \times \ell}$;
maximum iterations $T$.

\item[\textbf{Output:}]
Sensor selection sets (pair) $[\Sc^\star,\Se^\star]$.
\vspace{2pt}

\item[\textbf{Phase I: Candidate Budget Allocations} (see \Cref{iter_phase_I})]
\STATE Initialize $\mathcal{K} \gets \emptyset$
\PARFOR{$\ke = 0$ \textbf{to} $\lfloor \bud / c_{\exp} \rfloor$}
    \STATE $\kc \gets \left\lfloor (\bud - c_{\exp}\ke)/c_{\ch} \right\rfloor$
    \IF{$(\kc,\ke)$ is admissible (cf.~\Cref{prop:monotone,lower_variance_sensor})}
        \STATE $\mathcal{K} \gets \mathcal{K} \cup \{(\kc,\ke)\}$
    \ENDIF
\ENDPARFOR

\vspace{4pt}
\item[\textbf{Phase II: Iterative Sensor Placement} (see \Cref{iter_phase_II})]
\STATE Initialize design set $\mathcal{S}_{\mathcal{K}} \gets \emptyset$
\PARFOR{$(\kc,\ke) \in \mathcal{K}$}
    \STATE Initialize $\Sc \gets \emptyset$
    \STATE Initialize $\Se \gets \textsc{GreedySelect}(\exp,\ke,\Sc)$
    \FOR{$t = 1$ \textbf{to} $T$}
        \STATE $\Sc^{+} \gets \textsc{GreedySelect}(\ch,\kc,\Se)$
        \IF{$\Phi_D([\Sc^{+},\Se]) \le \Phi_D([\Sc,\Se])$}
            \STATE \textbf{break}
        \ENDIF
        \STATE $\Sc \gets \Sc^{+}$

        \STATE $\Se^{+} \gets \textsc{GreedySelect}(\exp,\ke,\Sc)$
        \IF{$\Phi_D([\Sc,\Se^{+}]) \le \Phi_D([\Sc,\Se])$}
            \STATE \textbf{break}
        \ENDIF
        \STATE $\Se \gets \Se^{+}$
    \ENDFOR
    \STATE $\mathcal{S}_{\mathcal{K}} \gets \mathcal{S}_{\mathcal{K}} \cup \{(\Sc,\Se)\}$
\ENDPARFOR

\vspace{4pt}
\STATE $\displaystyle(\Sc^\star,\Se^\star)
\gets \argmax_{(\Sc,\Se)\in \mathcal{S}_{\mathcal{K}}}
\Phi_D([\Sc,\Se])$
\RETURN $[\Sc^\star,\Se^\star]$
\end{algorithmic}
\end{algorithm}

\section{Relation between sensor fidelities and greedy selection}\label{sec:heuristic}

To explain the empirical preference of the greedy algorithm for one sensor fidelity over another, we derive a heuristic condition relating sensor costs and measurement noise levels. The key idea is to examine how the cost-normalized marginal information gain balances fidelity and noise.
We begin by writing the greedy selection rule in its cost-normalized form. Omitting the iteration counter for clarity, at each step, the algorithm selects the fidelity-location pair
\[(j^*,i^*) \in \argmax_{(j,i) \in J \times I} \left\{c_j^{-1}\left(\Phi_D(S^{(j,i)})-\Phi_D(S)\right)\right\}\]
where $S = [\Sc,\Se]$ is the current selection list, $J$ is the set of fidelities for which there is enough remaining budget, and $I = \{1,2,\hdots,M\} \setminus (\Sc \cup\Se)$ is the set of locations where there is no sensor. Notice that this is equivalent to first setting 
\begin{align*}
 i^*_\ch \in \argmax_{i\in I} \left\{\Phi_D(S^{(\ch,i)})\right\} \quad \mbox{and} \quad i^*_\exp \in \argmax_{i\in I} \left\{\Phi_D(S^{(\exp,i)})\right\},
\end{align*}
then setting
\begin{align*}
j^* \in \argmax_{j\in J} \left\{c_j^{-1}\left(\Phi_D(S^{(j,i_j^*)})-\Phi_D(S)\right) \right\}
\end{align*}
and $i^* \gets i^*_{j^*}$. That is, it is equivalent to compute the best location $i_\ch^*$ for a cheap sensor and the best location $i_\exp^*$ for an expensive one, then choose the better action between placing a cheap sensor at location $i_\ch^*$ or placing an expensive sensor at location $i_\exp^*$. 
Notice that, by the definition of $i_\ch^*$ and $i_\exp^*$, we must have
\[\Phi_D\left(S^{(\ch,i_\exp^*)}\right)  \leq \Phi_D\left(S^{(\ch,i_\ch^*)}\right) \quad \mbox{and} \quad \Phi_D\left(S^{(\exp,i_\ch^*)}\right)  \leq \Phi_D\left(S^{(\exp,i_\exp^*)}\right).\]
As noted in the proof of \Cref{lower_variance_sensor}, by the definition of $\bAc$ and $\bAe$ there exist some vectors $\bfac,\bfae$ such that $[\bA_{j}]_{:i_\ch^*} = \sigma_j^{-1} \bfac$ and $[\bA_{j}]_{:i_\exp^*} = \sigma_j^{-1} \bfae$ for each $j \in \{\ch,\exp\}$. 
Taking the matrix $\bB = \bB(S)$ from \Cref{prop:I+AAT}, applying \Cref{lemma:matrix-det-lemma} to these yields
\begin{align*}
    \log\left(1 + \sigmaC^{-2}\bfae\t \bB^{-1} \bfae \right) &\leq \log\left(1 + \sigmaC^{-2}\bfac\t \bB^{-1} \bfac \right), \\
    \log\left(1 + \sigmaE^{-2}\bfac\t \bB^{-1} \bfac \right) &\leq \log\left(1 + \sigmaE^{-2}\bfae\t \bB^{-1} \bfae \right).
\end{align*}
Together, these imply $\bfae\t \bB^{-1} \bfae \leq \bfac\t \bB^{-1} \bfac \leq \bfae\t \bB^{-1} \bfae$, so we must have
\[\bfac\t \bB^{-1} \bfac = \bfae\t \bB^{-1} \bfae.\]
Given this fact, we can now derive the aforementioned conditions. The greedy algorithm is indifferent between placing a cheap or expensive sensor when the marginal improvement from these two options is equal, i.e.,
\[\cC^{-1} (\Phi_D(S^{(\ch,i_\ch^*)}) - \Phi_D(S)) = \cE^{-1} (\Phi_D(S^{(\exp,i_\exp^*)}) - \Phi_D(S)).\]
That is, applying \Cref{prop:I+AAT} and \Cref{lemma:matrix-det-lemma}, the greedy algorithm considers it equally good to place a cheap sensor or an expensive sensor when
\[\cC^{-1} \log(1+\sigmaC^{-2} \bfac\t \bB^{-1} \bfac) = \cE^{-1} \log(1+\sigmaE^{-2} \bfae\t \bB^{-1} \bfae).\]
As we found above that $\bfac\t \bB^{-1} \bfac = \bfae\t \bB^{-1} \bfae$, this becomes
\[\cC^{-1} \log(1+\sigmaC^{-2} \bfac\t \bB^{-1} \bfac) = \cE^{-1} \log(1+\sigmaE^{-2} \bfac\t \bB^{-1} \bfac).\]
Whether or not this equality holds depends on the value of $\bfac\t \bB^{-1} \bfac$, which changes at each iteration, and thus one cannot find suitable conditions on the costs and measurement noise standard deviations such that the above necessarily holds at every iteration. 

However, this difficulty can be avoided by using a first-order Taylor approximation of the function 
\(x \mapsto \log(1+x)\) centered at \(x=0\). Since \(\log(1+x) = x + \bigO{x^2}\) as $x \to 0$, we have
\[c_j^{-1} \log(1+\sigma_j^{-2} \bfac\t \bB^{-1} \bfac) \approx c_j^{-1}\sigma_j^{-2} \bfac\t \bB^{-1} \bfac\]
for $j \in \{\ch,\exp\}$. Thus, the greedy algorithm is approximately indifferent between cheap and expensive sensors when 
\[\cC^{-1}\sigmaC^{-2} \bfac\t \bB^{-1} \bfac = \cE^{-1}\sigmaE^{-2} \bfac\t \bB^{-1} \bfac,\]
which is equivalent to
\[\frac{\cC}{\cE} = \frac{\sigmaE^2}{\sigmaC^2}.\]
Replacing this equality by a strict inequality yields an approximate preference for one fidelity over the other; namely, the greedy algorithm prefers expensive sensors when ${\cC}/{\cE} > {\sigmaE^2}/{\sigmaC^2}$ and cheap sensors when ${\cC}/{\cE} < {\sigmaE^2}/{\sigmaC^2}$.

\section{Extension of iterative algorithm to an arbitrary number of fidelities}
\label{sec:appendix_extension}
We extend the iterative algorithm from \Cref{sec:iterative} to the case in which one has an arbitrary number of sensor fidelities, labeled by $\{1,2,\hdots,f\}$, with associated measurement noise standard deviations $\{\sigma_1,\sigma_2,\hdots,\sigma_f\}$ and costs $\{c_1,c_2,\hdots,c_f\}$, where $\sigma_1 > \sigma_2 > \cdots > \sigma_f$ and $c_1 < c_2 < \cdots < c_f$, and one aims to find a sensor configuration $S = [S_1,S_2,\hdots,S_f]$ with maximal D-optimality value subject to the constraint that $\sum_{j=1}^f c_j k_j \leq \bud$, where $k_j = |S_j|$.

In Phase 1, construct the set of candidate allocations by considering each fixed feasible value of $(k_2,k_3,\hdots,k_f)$, then considering the maximal value of $k_1$ subject to this $(k_2,k_3,\hdots,k_f)$; if the resulting allocation $(k_1,k_2,\hdots,k_f)$ does not leave enough leftover budget to replace a sensor by a lower-variance alternative, add it to the set of candidates. That is, initialize $\mathcal K \gets \emptyset$, then, for each tuple of nonnegative integers $(k_2,k_3,\hdots,k_f)$ such that $\sum_{j=2}^f c_j k_j \le \bud$ (which can be computed by a nested loop over the feasible values of $k_j$ for each $j \in \{f, f-1, \hdots, 2\}$), consider the maximal nonnegative integer $k_1$ such that $\sum_{j=1}^f c_j k_j \le \bud$, and check that $(k_1,k_2,\hdots,k_f)$ satisfies that for each $i \in \{1,2,\hdots,f-1\}$ one has either $k_i = 0$ or $\sum_{j=1}^f c_j k_j > \bud - (c_{i+1} - c_i)$; if this condition is satisfied, then add $(k_1,k_2,\hdots,k_f)$ to $\mathcal K$.

Then, in Phase 2, for each $(k_1,k_2,\hdots,k_f) \in \mathcal K$, carry out the iterative greedy optimization in the same fashion as in \Cref{iter_phase_II}, proceeding through the different fidelities in decreasing order of their cost, terminating the iterations after a maximum number of iterations or if the objective function decreases. That is, initialize $S_j^{(0)} \gets \emptyset$ for each $j \in \{1,2,\hdots,f\}$. Then, for $t \geq 1$, for $j = f, f-1, f-2, \hdots, 1$, obtain $S_j^{(t)}$ by greedily maximizing the value of the function \[S_j \mapsto \Phi_D\left(\left[S_1^{(t-1)}, S_2^{(t-1)}, \hdots, S_{j-1}^{(t-1)}, S_j, S_{j+1}^{(t)}, S_{j+2}^{(t)}, \hdots, S_f^{(t)}\right]\right)\]
subject to the constraints that $|S_j|=k_j$ and $S_j \cap S_i = \emptyset$ for each $i \neq j$. For $t \geq 1$ and $j \in \{1,2,\hdots,f\}$, terminate after obtaining $S_j^{(t)}$ if either the maximum number $T$ of iterations has been reached or if the objective function decreased when $S_j^{(t-1)}$ was replaced by $S_j^{(t)}$.

To analyze the computational cost of this process, observe that the number of candidate allocations in $\mathcal K$ is upper bounded by $\prod_{j=2}^f (1+\bud/c_j)$, so the total cost of Phase 1 is $\bigO{\frac{b^{f-1}}{\prod_{j=2}^f c_j}}$ flops. In Phase 2, the total cost for processing a single candidate $(k_1,k_2,\hdots,k_f)$ is $\bigO{\left(\sum_{j=1}^f k_j\right) TM\ell}$ flops, and since all admissible allocations satisfy $\sum_{j=1}^f k_j \leq \bud/c_1$, the cost per candidate is $\bigO{\left(\bud/c_1\right) TM\ell}$. Since there are $\bigO{\frac{b^{f-1}}{\prod_{j=2}^f c_j}}$ candidates, the total cost of Phase 2 is $\bigO{\frac{b^{f}}{\prod_{j=1}^f c_j}TM\ell}$ flops. This term dominates the cost of Phase 1, so the total cost of the extension of this algorithm to the $f$-fidelity case is $\bigO{\frac{b^{f}}{\prod_{j=1}^f c_j}TM\ell}$ flops.

\section*{Acknowledgments}
This work was supported in part by the National Science Foundation under Awards
DMS-2349611 and DMS-2411198, and by the U.S. Department of Energy, Of{}fice of Science, Advanced Scientific Computing Research program under Awards DE-SC0023188 and DE-SC0026310.


\medskip
Received xxxx 20xx; revised xxxx 20xx; early access xxxx 20xx.
\medskip


\begin{thebibliography}{10}

\bibitem{alexanderian2023briefnotebayesiandoptimality}
{\sc A.~Alexanderian}, {\em {A brief note on the Bayesian D-optimality criterion}}, 2023, \url{https://arxiv.org/abs/2212.11466}, \url{https://arxiv.org/abs/2212.11466}.

\bibitem{Attia_2022}
{\sc A.~Attia, S.~Leyffer, and T.~S. Munson}, {\em {Stochastic Learning Approach for Binary Optimization: Application to {B}ayesian Optimal Design of Experiments}}, SIAM Journal on Scientific Computing, 44 (2022), p.~B395–B427, \url{https://doi.org/10.1137/21m1404363}, \url{http://dx.doi.org/10.1137/21M1404363}.

\bibitem{mri_imaging}
{\sc T.~Bakker, H.~van Hoof, and M.~Welling}, {\em {Experimental design for MRI by greedy policy search}}, in Proceedings of the 34th International Conference on Neural Information Processing Systems, NIPS '20, Red Hook, NY, USA, 2020, Curran Associates Inc.

\bibitem{fluid2}
{\sc A.~Barklage, M.~Stradtner, and P.~Bekemeyer}, {\em {Sensor placement for optimal aerodynamic data fusion}}, Aerospace Science and Technology, 155 (2024), p.~109598, \url{https://doi.org/https://doi.org/10.1016/j.ast.2024.109598}, \url{https://www.sciencedirect.com/science/article/pii/S1270963824007272}.

\bibitem{evaluation_metrics}
{\sc S.~Chakkor, E.~Cheikh, B.~Mostafa, and A.~Hajraoui}, {\em {Efficiency Evaluation Metrics for Wireless Intelligent Sensors Applications}}, International Journal of Intelligent Systems and Application, 6 (2014), pp.~1--10, \url{https://doi.org/10.5815/ijisa.2014.10.01}.

\bibitem{9152984}
{\sc E.~Clark, S.~L. Brunton, and J.~N. Kutz}, {\em {Multi-Fidelity Sensor Selection: Greedy Algorithms to Place Cheap and Expensive Sensors With Cost Constraints}}, IEEE Sensors Journal, 21 (2021), pp.~600--611, \url{https://doi.org/10.1109/JSEN.2020.3013094}.

\bibitem{eswar2025bayesiandoptimalexperimentaldesigns}
{\sc S.~Eswar, V.~Rao, and A.~K. Saibaba}, {\em {Bayesian {D}-Optimal Experimental Designs via Column Subset Selection}}, 2025, \url{https://arxiv.org/abs/2402.16000}, \url{https://arxiv.org/abs/2402.16000}.

\bibitem{tsunami}
{\sc A.~Ferrolino, J.~E. Lope, and R.~Mendoza}, {\em {Optimal Location of Sensors for Early Detection of Tsunami Waves}}, Springer International Publishing, 06 2020, pp.~562--575, \url{https://doi.org/10.1007/978-3-030-50417-5_42}.

\bibitem{golub2013matrix}
{\sc G.~H. Golub and C.~F.~V. Loan}, {\em {Matrix Computations}}, Johns Hopkins University Press, Baltimore, MD, 4th~ed., 2013.

\bibitem{Guenther17}
{\sc T.~G{\"u}nther, M.~Gross, and H.~Theisel}, {\em {Generic Objective Vortices for Flow Visualization}}, ACM Transactions on Graphics (Proc. SIGGRAPH), 36 (2017), pp.~141:1--141:11.


\bibitem{article}
{\sc E.~Haber, L.~Horesh, and L.~Tenorio}, {\em {Numerical methods for experimental design of large-scale linear ill-posed inverse problems}}, Inverse Problems, 24 (2008), p.~055012, \url{https://doi.org/10.1088/0266-5611/24/5/055012}.
\bibitem{Hansen}
{\sc P.~C. Hansen}, {\em {Discrete Inverse Problems: Insight and Algorithms}}, vol.~7 of Fundamentals of Algorithms, Society for Industrial and Applied Mathematics (SIAM), Philadelphia, PA, 2010, \url{https://doi.org/10.1137/1.9780898718836}.
\bibitem{HornJohnson2012}
{\sc R.~A. Horn and C.~R. Johnson}, {\em {Matrix Analysis}}, Cambridge University Press, Cambridge, UK; New York, NY, USA, 2nd~ed., 2012.

\bibitem{kakasenko2025bridginggapdeterministicprobabilistic}
{\sc L.~Kakasenko, A.~Alexanderian, M.~Farazmand, and A.~K. Saibaba}, {\em {Bridging the Gap Between Deterministic and Probabilistic Approaches to State Estimation}}, 2025, \url{https://arxiv.org/abs/2505.04004}, \url{https://arxiv.org/abs/2505.04004}.

\bibitem{kandasamy}
{\sc K.~Kandasamy, G.~Dasarathy, J.~Oliva, J.~Schneider, and B.~P\'{o}czos}, {\em {Gaussian process bandit optimisation with multi-fidelity evaluations}}, in Proceedings of the 30th International Conference on Neural Information Processing Systems, NIPS'16, Red Hook, NY, USA, 2016, Curran Associates Inc., p.~1000–1008.

\bibitem{Khuller}
{\sc S.~Khuller, A.~Moss, and J.~S. Naor}, {\em {The budgeted maximum coverage problem}}, Inf. Process. Lett., 70 (1999), p.~39–45, \url{https://doi.org/10.1016/S0020-0190(99)00031-9}, \url{https://doi.org/10.1016/S0020-0190(99)00031-9}.

\bibitem{Korte2008}
{\sc B.~Korte and J.~Vygen}, {\em The Knapsack Problem}, Springer Berlin Heidelberg, Berlin, Heidelberg, 2008, pp.~439--448, \url{https://doi.org/10.1007/978-3-540-71844-4_17}, \url{https://doi.org/10.1007/978-3-540-71844-4_17}.

\bibitem{JMLR:v9:krause08a}
{\sc A.~Krause, A.~Singh, and C.~Guestrin}, {\em {Near-Optimal Sensor Placements in Gaussian Processes: Theory, Efficient Algorithms and Empirical Studies}}, Journal of Machine Learning Research, 9 (2008), pp.~235--284, \url{http://jmlr.org/papers/v9/krause08a.html}.

\bibitem{biosensors}
{\sc B.~{Kumar Kundu}}, {\em {Chapter 1 - Introduction to sensors and types of biosensors}}, in {Multifaceted Bio-sensing Technology}, L.~Singh, D.~Mahapatra, and S.~Kumar, eds., vol.~4 of Bioelectrochemical Systems: The way forward, Academic Press, 2023, pp.~1--12, \url{https://doi.org/https://doi.org/10.1016/B978-0-323-90807-8.00002-6}, \url{https://www.sciencedirect.com/science/article/pii/B9780323908078000026}.

\bibitem{Leskovec}
{\sc J.~Leskovec, A.~Krause, C.~Guestrin, C.~Faloutsos, J.~Vanbriesen, and N.~Glance}, {\em {Cost-effective outbreak detection in networks}}, Proceedings of the ACM SIGKDD International Conference on Knowledge Discovery and Data Mining, 420-429 (2007), pp.~420--429, \url{https://doi.org/10.1145/1281192.1281239}.

\bibitem{structural_damage}
{\sc J.-F. Lin, Y.-L. Xu, and S.-S. Law}, {\em {Structural damage detection-oriented multi-type sensor placement with multi-objective optimization}}, Journal of Sound and Vibration, 422 (2018), pp.~568--589, \url{https://doi.org/https://doi.org/10.1016/j.jsv.2018.01.047}, \url{https://www.sciencedirect.com/science/article/pii/S0022460X18300695}.

\bibitem{ocean_climate_observing}
{\sc N.~Loose and P.~Heimbach}, {\em {Leveraging Uncertainty Quantification to Design Ocean Climate Observing Systems}}, Journal of Advances in Modeling Earth Systems, 13 (2021), p.~e2020MS002386, \url{https://doi.org/https://doi.org/10.1029/2020MS002386}, \url{https://agupubs.onlinelibrary.wiley.com/doi/abs/10.1029/2020MS002386}.

\bibitem{Marcus1965ASO}
{\sc M.~Marcus and H.~Minc}, {\em {A Survey of Matrix Theory and Matrix Inequalities}}, The American Mathematical Monthly,  (1965), \url{https://api.semanticscholar.org/CorpusID:121044889}.

\bibitem{Minoux}
{\sc M.~Minoux}, {\em {Accelerated greedy algorithms for maximizing submodular set functions}}, in Optimization Techniques, J.~Stoer, ed., Berlin, Heidelberg, 1978, Springer Berlin Heidelberg, pp.~234--243.

\bibitem{noaa_oisst_v2}
{\sc {National Oceanic and Atmospheric Administration (NOAA)}}, {\em {NOAA Optimal Interpolation (OI) Sea Surface Temperature (SST) v2}}.
\newblock \url{https://www.ncei.noaa.gov/products/optimum-interpolation-sst}, 2021.
\newblock Accessed: 2025-06-17.

\bibitem{nemhauser1978analysis}
{\sc G.~L. Nemhauser, L.~A. Wolsey, and M.~L. Fisher}, {\em {An analysis of approximations for maximizing submodular set functions—I}}, Mathematical Programming, 14 (1978), pp.~265--294, \url{https://doi.org/10.1007/BF01588971}.

\bibitem{fluid1}
{\sc T.~Nishida, N.~Ueno, S.~Koyama, and H.~Saruwatari}, {\em {Region-Restricted Sensor Placement Based on Gaussian Process for Sound Field Estimation}}, IEEE Transactions on Signal Processing, 70 (2022), pp.~1718--1733, \url{https://doi.org/10.1109/TSP.2022.3156012}.


\bibitem{cvxrelax}
{\sc J.~Paredes-Ahumada, P.~Ferrer-Cid, J.~M. Barcelo-Ordinas, and J.~Garcia-Vidal}, {\em {Convex Relaxation Method for Sensor Placement in Multiclass Monitoring Networks}}, IEEE Transactions on Instrumentation and Measurement, 73 (2024), pp.~1--13, \url{https://doi.org/10.1109/TIM.2024.3461788}.

\bibitem{Robertazzi}
{\sc T.~G. Robertazzi and S.~C. Schwartz}, {\em {An Accelerated Sequential Algorithm for Producing D-Optimal Designs}}, SIAM Journal on Scientific and Statistical Computing, 10 (1989), pp.~341--358, \url{https://doi.org/10.1137/0910022}, \url{https://doi.org/10.1137/0910022}, \url{https://arxiv.org/abs/https://doi.org/10.1137/0910022}.

\bibitem{5717225}
{\sc M.~Shamaiah, S.~Banerjee, and H.~Vikalo}, {\em {Greedy sensor selection: Leveraging submodularity}}, in 49th IEEE Conference on Decision and Control (CDC), 2010, pp.~2572--2577, \url{https://doi.org/10.1109/CDC.2010.5717225}.

\bibitem{SVIRIDENKO200441}
{\sc M.~Sviridenko}, {\em {A note on maximizing a submodular set function subject to a knapsack constraint}}, Operations Research Letters, 32 (2004), pp.~41--43, \url{https://doi.org/https://doi.org/10.1016/S0167-6377(03)00062-2}, \url{https://www.sciencedirect.com/science/article/pii/S0167637703000622}.

\end{thebibliography}
\end{document}